\documentclass[a4paper]{article}
\RequirePackage[OT1]{fontenc}
\RequirePackage{amsthm,amsmath}
\RequirePackage[numbers]{natbib}
\usepackage[colorlinks,citecolor=blue,urlcolor=blue,linkcolor=blue,linktocpage=true]{hyperref}

\usepackage{bbm}
\usepackage{amssymb}
\usepackage{graphicx}

\newcommand{\R}{{\mathbb R}}

\newcommand{\sgn}{\mathrm{sgn}}
\newcommand{\cotan}{\mathrm{cotan}}
\newcommand{\var}{\mathrm{Var}}
\newcommand{\E}{\mathbb E}
\newcommand{\poinc}{C_{\mathrm{P}}}
\newcommand{\sg}{\lambda}
\newcommand{\pdf}{\rho}
\newcommand{\Leb}{\mathrm{Leb}}
\newcommand{\truncnorm}{\mathcal{N}_{\vert [a,b]}}
\newcommand{\Sob}[2]{\mathcal{H}^{#1}_{#2}}   
\newcommand{\SobCent}[2]{\mathcal{H}^{#1, \mathrm{c}}_{#2}}   

\numberwithin{equation}{section}
\theoremstyle{plain}
\newtheorem{thm}{Theorem}
\newtheorem{lem}{Lemma}
\newtheorem{defi}{Definition}
\newtheorem{prop}{Proposition}
\newtheorem{cor}{Corollary}
\newtheorem{rem}{Remark}

\newcommand{\mb}[1]{\mathbf{#1}}

\let\oldtabular\tabular 
\renewcommand{\tabular}{\footnotesize\oldtabular}
\usepackage{caption}
\usepackage[small]{titlesec}

\begin{document}

\title{\bf{Poincar\'e inequalities on intervals --\\ application to sensitivity analysis}}
\footnotetext[1]{Mines Saint-Etienne, UMR CNRS 6158, LIMOS, F--42023 Saint-\'{E}tienne, France -- \href{mailto:roustant@emse.fr}{roustant@emse.fr}}
\footnotetext[2]{Institut de Math\'ematiques de Toulouse, Universit\'e Paul Sabatier, 31062 Toulouse Cedex 9, France -- \href{mailto:franck.barthe@math.univ-toulouse.fr}{franck.barthe@math.univ-toulouse.fr}}
\footnotetext[3]{Electricit\'e de France R\&D, 6 quai Watier, Chatou, F--78401, France -- \href{mailto:bertrand.iooss@edf.fr}{bertrand.iooss@edf.fr}}
\author{\renewcommand{\thefootnote}{\arabic{footnote}}
\large
Olivier Roustant\footnotemark[1] ,
~Franck Barthe\footnotemark[2] ,
~Bertrand Iooss\footnotemark[2]~$^,$\footnotemark[3]}

\maketitle

\begin{abstract}
The development of global sensitivity analysis of numerical model outputs has recently raised new issues on 1-dimensional Poincar\'e inequalities. Typically two kind of sensitivity indices are linked by a Poincar\'e type inequality, which provide upper bounds of the most interpretable index by using the other one, cheaper to compute. This allows performing a low-cost screening of unessential variables. The efficiency of this screening then highly depends on the accuracy of the upper bounds in Poincar\'e inequalities.\\
The novelty in the questions concern the wide range of probability distributions involved, which are often truncated on intervals. After providing an overview of the existing knowledge and techniques, we add some theory about Poincar\'e constants on intervals, with improvements for symmetric intervals. Then we exploit the spectral interpretation for computing exact value of Poincar\'e constants of any admissible distribution on a given interval. We give semi-analytical results for some frequent distributions (truncated exponential, triangular, truncated normal), and present a numerical method in the general case.\\
Finally, an application is made to a hydrological problem, showing the benefits of the new results in Poincar\'e inequalities to sensitivity analysis.\\
\end{abstract}

\noindent {\small \begin{bf}Keywords: \end{bf}
Poincar\'e inequality, Spectral gap, Truncated distribution, Kummer's functions, Sobol-Hoeffding decomposition, Sobol indices, Derivative-based global sensitivity measures, Finite elements.}\\

\newpage 

\tableofcontents
\bigskip

\section{Introduction} 

\subsection{Motivation}
This research is motivated by sensitivity analysis of a costly numerical function $f(x_1, \dots, x_d)$ of independent random variables $x_i$. 
We denote by $\mu = \mu_1 \otimes \dots \otimes \mu_d$ the distribution of $\mathbf{x} = (x_1, \dots, x_d)$ and assume that $f(\mathbf{x}) \in L^2(\mu)$.
Such problem occurs for instance in computer experiments, where a physical phenomenon is studied with a complex numerical code (\cite{DeRocquigny_et_al_2008}).
One important question is \textit{screening}, i.e. to identify unessential variables, which can be done by computing several criteria, called sensitivity indices.
Among them, the variance-based indices -- or \emph{Sobol indices}, are often prefered by practitioners due to their easy interpretation (\cite{Iooss_Lemaitre_review}).
They are defined on the Sobol-Hoeffding decomposition (\cite{Hoeffding_1948}, \cite{Efron_Stein_1981}, \cite{Sobol_1993}),
\begin{equation*}
f(\mb{x}) = \sum_{I \subseteq \{1, \dots, d\}} f_I(\mb{x}_I)
= f_0 + \sum_{1 \leq i \leq d} f_i(x_i) 
+ \sum_{1 \leq i<j \leq d} f_{i,j}(x_i, x_j) 
+ \dots 
\end{equation*}
where $\mb{x}_I$ is the subvector extracted from $\mb{x}$ whose coordinates belong to $I$. 
The terms $f_I$ satisfy non-simplification conditions
$ E(f_I(\mb{x}_I) \vert \mb{x}_J) = 0$ for all strict subset $J \subset I$, equivalent to: 
\begin{equation} \label{eq:SobolNonSimplificationCondition}
\int f_I(\mb{x}_I) d\mu_i(\mb{x}_i) = 0
\end{equation}
for all $i \in I$ and all $\mb{x}_I$. Such conditions imply orthogonality, and lead to the variance decomposition: 
$$\var (f(\mb{x})) = \sum_{I \subseteq \{1, \dots, d\}} \var( f_I(\mb{x}_I) )$$
The ``total Sobol index'' $S_i^T$ is then defined as the ratio of variance of $f(\mathbf{x})$ explained by $x_i$ (potentially with other variables):
$$S_i^T = \var \left( \sum_{I \ni i} f_I(\mb{x}_I) \right) / D$$
where $D = \var(f(\mb{x}))$. Thus one can decide that $x_i$ is not influential if $S_i^T$ is less than (say) $5\%$.\\

Despite their nice interpretation, Sobol indices require numerous computations for their estimation. Then another global sensitivity index, called DGSM (Derivative-based Global Sensitivity Measure), can advantageously be used as a proxy \cite{Sobol_2009,Kucherenko_et_al_2009}. Defined by 
$$ \nu_i = \int \left( \frac{\partial f(\mb{x})}{\partial x_i} \right)^2 d\mu(\mb{x})$$ 
for $i=1,\dots,d$, they are cheaper to compute, especially when the gradient of $f$ is available (e.g. as output of an evaluation of a complex numerical code). As shown in \cite{Lamboni_et_al_2013}, Sobol indices and DGSM are connected by a 1-dimensional Poincar\'e-type inequality:
\begin{equation} \label{eq:PoincareTypeIneq}
S_i^T \leq C(\mu_i) \nu_i / D
\end{equation}
where $C(\mu_i)$ is a Poincar\'e constant for $\mu_i$. 
Recall that $\mu_i$ satisfies a Poincar\'e inequality if the energy of any centered function is controlled by the energy of its derivative: For all $g$ satisfying $\int g d\mu_i = 0$ there exists $C(\mu_i) > 0$ s.t.
\begin{equation} \label{eq:PoincareInequality}
\int g^2 d\mu_i \leq C(\mu_i) \int g'^2 d\mu_i
\end{equation}
Inequality~\ref{eq:PoincareTypeIneq} is easily obtained by applying the Poincar\'e inequality~\ref{eq:PoincareInequality} to the functions $x_i \mapsto \sum_{I \ni i} f_I(\mb{x}_I)$, which are centered for any choice of $\mb{x}_I$ with $i \in I$ (condition~\ref{eq:SobolNonSimplificationCondition}), 
and integrating with respect to the other variables $x_j$ ($j \neq i$).
This allows performing a ``low-cost'' screening based on the upper bound of Sobol indices in \ref{eq:PoincareTypeIneq} instead of Sobol indices directly.
Hence, one can decide that $x_i$ is not influential if 
$C(\mu_i) \frac{\nu_i}D$ is less than (say) $5\%$.\\

The efficiency of this low-cost screening strongly depends on the accuracy of the Poincar\'e-type inequality \ref{eq:PoincareTypeIneq}, and this motivates the investigation of 1-dimensional Poincar\'e inequalities. Notice that there is a large variety of probability distributions used in sensitivity analysis. They are often linked to prior knowledge about the range of variation of physical parameters. As a consequence most of them are continuous with finite support, possibly obtained by truncation. The most frequent probability density function (pdf) are: Uniform, (truncated) Gaussian, triangular, (truncated) lognormal, (truncated) exponential, (truncated) Weibull, (truncated) Gumbel (see for example \cite{DeRocquigny_et_al_2008}). Less frequently, it can be found: (inverse) Gamma, Beta, trapezoidal, generalized extreme value.\\

It is also important to notice that one may not easily reduce the problem to the case of the uniform distribution, by applying the standard reincreasing arrangement technique. 
Indeed by denoting $F_i$ the cdf of $\mu_i$, an idea would be to consider the function
$g(u_1, \dots, u_d) = f \left( F_1^{-1}(u_1), \dots, F_d^{-1}(u_d) \right)$ where the $u_i = F_i(x_i)$ are independent and uniform on $[0,1]$. Then the total Sobol indices of $f(x_1, x_2)$ and $g(u_1, u_2)$ are equal.
However, the derivatives of the transformations $F_i^{-1}$ can be large, and the DGSM computed on $g$ may be large and even infinite. For instance, if $f(x_1, x_2) = x_1 + x_2$ with $x_1, x_2$ i.i.d. $\mathcal{N}(0,1)$, the DGSM of $g(u_1, u_2) = \Phi^{-1}(u_1) + \Phi^{-1}(u_2)$ is equal, for each variable $u_i$, to:\\ 
$$\nu_g = \int_0^1 [(\Phi^{-1})'(s)]^2 ds = 
\int_0^1 \frac{1}{\phi (\Phi^{-1}(s))^2} ds = 
\int_{\R} \frac{1}{\phi(t)} dt = \int_{\R} e^{t^2/2}dt = +\infty$$
On the other hand, the DGSM $\nu_f$ of $f$ is equal to $1$ for each $x_i$, and the corresponding upper bound is $\poinc (\mathcal{N}(0,1))\frac{\nu_f}{D} = \frac12$, 
which is here exactly equal to the total Sobol index of $x_i$.\\

To conclude about motivations in sensitivity analysis, it is worth mentioning that the idea of low-cost screening can be extended to higher-order interactions, and also depends on the accuracy of 1-dimensional Poincar\'e inequalities  \cite{RoustantCrossDerivative}. It allows screening out useless interactions and discovering additive structures in $f$.\\

\subsection{Aim and plan}
Our aim is to bridge the gap between industrial needs coming from sensitivity analysis problems and the theory of Poincar\'e inequalities, by providing an accessible introduction to Poincar\'e inequalites for non-specialists, and by using and developping the theory in order to deal with 
the specific situations motivated by low-cost screening.

In Section~\ref{sec:background}, we present the general background on Poincar\'e inequalities, 
together with the main techniques avaible to establish them (Muckenhoupt's criterion, perturbation and transportation methods, spectral methods). 
 Most of these techniques provide upper bounds on the Poincar\'e constant for large classes of measures. 

In the literature of 1-dimensional Poincar\'e inequalities, the exact value of the Poincar\'e constant is known for very few measures, such as: Uniform, Gaussian (which are classical, see e.g. \cite{Ane_2000}), exponential \cite{Bobkov_Ledoux} or logistic \cite{Barthe_logistic}. More measures are needed for sensitivity analysis, as well as their restrictions to intervals. These situations were not studied specifically so far (apart from the Gaussian pdf, for which optimal constants are known on the whole real line and on intervals formed by successive zeros of Hermite polynomials, see e.g. \cite{Dautray_Lions_vol3}). Section~\ref{sec:generalResults} provides further results which allow to deal with the family of restrictions of a given probability measure to intervals. In the case of symmetric measures and intervals, we derive new improvements.

Section~\ref{sec:optimalConstants} deals with the exact constant for some distributions required by sensitivity analysis.
We present new semi-analytical results for the optimal constant of the triangular, truncated normal and truncated exponential. The optimal constant is obtained as the first zero of an analytic function.
We also provide a general numerical algorithm that converges to the optimal Poincar\'e constant.

Finally we come back to our initial motivation and apply this research to two test-cases (Section~\ref{sec:applications}). 
It is observed that the new results significantly improve on the exisiting ones, 
and now allow performing a low-cost screening of Sobol indices based on upper bounds computed by DGSMs.

For the ease of readability, the most technical proofs are postponed to the Appendix.

\section{Background on Poincar\'e inequalities} \label{sec:background}

In this section, we provide a quick survey of the main simple techniques allowing to derive Poincar\'e inequalities for probability measures on the real line. We often make regularity assumptions on the measures. This allows to avoid  technicalities, without reducing the scope for realistic applications. Indeed, all the measures we are interested in are
supported on an open interval $(a,b)$, and have a positive continuous, piece-wize continuouly differentiable density. 
Moreover, when $a$ is finite, they are monotonic on a neighborhood of $a$ (and similarly for $b$).

\subsection{Definitions} \label{sec:Definitions}

Consider an open interval of the real line $\Omega = (a,b)$ with $-\infty\le a<b\le +\infty$.
A locally integrable function $f:\Omega\to \mathbb R$ is weakly differentiable if there exists
a locally integrable function $g:\Omega\to \mathbb R$ such that for all functions $\phi$ of class $C^\infty$ with compact support in $\Omega$:
$$ \int_\Omega f(t) \phi'(t) dt= - \int_\Omega g(t) \phi(t) dt.$$
Then $g$ is a.e. uniquely determined  (more precisely two functions with this property coincide almost everywhere), it is called the weak derivative of $f$ and denoted by $f'$.

Let $\mu$ be a probability measure  on $\Omega$, and $f:\Omega\to \mathbb R$ be a  Borel measurable function. Recall that the variance of $f$ for $\mu$ is defined as
$$\var_\mu(f)=\inf_{a\in \mathbb R} \int_\Omega (f-a)^2 d\mu.$$
Obviously $\var_\mu(f)=+\infty$ if $\int f^2d\mu=+\infty$. When $\int f^2d\mu<+\infty$, it holds
$$\var_\mu(f)=\int f^2d\mu -\left( \int f\, d\mu\right)^2=\int \left(f-\int f\, d\mu \right)^2d\mu.$$

\begin{defi}\label{ded:poinc1}
Let  $\mu(dt)=\pdf(t)dt$ be an absolutely continuous probability measure $\Omega$. 
We say that $\mu$ verifies a \emph{Poincar\'e inequality} on $\Omega$ if there exists a constant $C<+\infty$ such that for all $f$ weakly differentiable functions $f$ on $\Omega$:
\begin{equation} \label{def:PoincareInequality0}
\var_\mu(f) \leq C \int_\Omega (f')^2 d\mu.
\end{equation} 
In this case, the smallest possible constant $C$ above is denote $\poinc(\mu)$, it is refered to as the Poincar\'e constant of the measure $\mu$.
\end{defi}
 
Observe that the above integrals are always defined, with values in $[0,+\infty]$. Roughly speaking,
a Poincar\'e inequality expresses in a quantitative way that
a function with a small weak derivative, measured in the sense of $\mu$, has to be close to a constant function again 
in the sense of $\mu$.

\begin{rem} \label{rem:caracWeekly}
Weakly differentiable functions are exactly the functions which admit (in their equivalence class for 
a.e. equality) a continuous version which satisfies (see e.g. \cite{Allaire_Book}, \cite{Brezis})
$$\forall x,y \in \Omega, \quad f(y) = f(x) + \int_x^y f'(t) \, dt.$$
Such functions are also called (locally) absolutely continuous. Their variations can be recovered by integrating their weak derivatives. It is therefore plain that they provide a good setting for Poincar\'e inequalities. On the contrary, it is not possible to 
work just with a.e. differentiable functions: for instance the famous Cantor function (a.k.a. the Devil's stairs) increases from 0 to 1 but is a.e. differentiable with zero derivative. However, everywhere differentiable functions with locally integrable derivative are weakly differentiable, see e.g. \cite{Rudin}.
\end{rem}

To be very precise, we should have denoted the Poincar\'e constant as $\poinc(\mu,\Omega)$.
Indeed, we could also consider the probability measure $\mu$ on $\Omega$ as acting on any larger open interval $\Omega'$. The restriction to $\Omega$ of weakly differentiable functions on $\Omega'$ are specific weakly differentiable function on $\Omega$ (they are continuous at boundary points). Therefore, satisfying a Poincar\'e inequality on $\Omega$ is formally a more demanding property. From now on, we will assume that $\Omega$ is the interior of the convex hull of the support of $\mu$, which is consistent with the notation $\poinc(\mu)$.
  
Obviously, some measures cannot satisfy a Poincar\'e inequalities, for instance the uniform measure on $(0,1)\cup(2,3)$ (one can choose a differentiable function $f$ on $(0,3)$ which is equal to 0 on $(0,1)$ and to 1 on $(2,3)$. Then $\int (f')^2 d\mu=0$).

\medskip
 
Next let us present an equivalent definition of Poincar\'e inequalities, in a more convenient analytic setting.
Let $L^2(\mu)$ be the set of (equivalence classes for a.e. equality of) measurable functions on $\Omega$  such that $\int_\Omega f^2 d\mu < +\infty$. 
We denote $\Vert . \Vert_{L^2(\mu)}$, or simply $\Vert . \Vert$, its associated norm:
$\Vert f \Vert = \left( \int_\Omega f^2 d\mu \right)^{1/2}$. 
The first weighted  Sobolev space $\Sob{1}{\mu}(\Omega)$ is defined by   
$$\Sob{1}{\mu}(\Omega) = \{ f \in L^2(\mu) \textrm{ such that } f' \in L^2(\mu) \},$$
where $f'$ is the weak derivative of $f$. 
It is well known that $\Sob{1}{\mu}(\Omega)$ is an Hilbert space with the norm $\Vert f \Vert_{\Sob{1}{\mu}(\Omega)}^2 = \Vert f \Vert^2 + \Vert f' \Vert^2$. 
More generally, for an integer $\ell \geq 1$, the space $\Sob{\ell}{\mu}(\Omega)$ is defined by:
$$\Sob{\ell}{\mu}(\Omega) = \{ f \in L^2(\mu) \textrm{ such that for all } k \leq \ell, f^{(k)} \in L^2(\mu)\}$$
where $f^{(k)}$ is the $k$-th weak derivative of $f$.
An important particular case is when $\Omega$ is bounded and there exists two positive constants $m, M$ such that
$$\forall t \in \overline{\Omega}, \quad 0 < m  \leq \pdf(t) \leq M$$
Then if $\Leb$ denotes the Lebesgue measure on $\Omega$, 
it holds that $L^2(\mu) = L^2(\Leb)$ and $\Sob{\ell}{\mu}(\Omega) = \Sob{\ell}{\Leb}(\Omega)$ for all positive integers $\ell$, and the norms on these spaces are equivalent when $\mu$ is replaced by $\Leb$. In other words, the weighted Sobolev spaces are equivalent to the usual Sobolev spaces. This remark will allow us to use several results which are available for $\Leb$ and a bounded $\Omega$. Firstly, $\Sob{1}{\mu}(\mu)$ then contains functions which are continuous on $\overline{\Omega}$ and piecewise $C^1$ on $\Omega$. 
Secondly, the spectral theory of elliptic problems (see e.g. \citep{Allaire_Book}, Chap. 7) will then be valid.\\
We now come back to the general case.

\begin{defi}\label{ded:poinc2}
Let  $\mu(dt)=\pdf(t)dt$ be an absolutely continuous probability measure on an open interval $\Omega$, such that $\pdf>0$ almost everywhere on $\Omega$.
We say that $\mu$ admits a \emph{Poincar\'e inequality} on $\Omega$ if there exists a constant $C<+\infty$ such that for all $f$ in $\Sob{1}{\mu}(\Omega)$ verifying $\int_\Omega f d\mu = 0$, we have:
\begin{equation} \label{def:PoincareInequality}
\int_\Omega f^2 d\mu \leq C \int_\Omega (f')^2 d\mu
\end{equation} 
The best possible constant $C$ is denoted  $\poinc(\mu)$. If there exists $f_{\mathrm{opt}}$, a centered function of $\Sob{1}{\mu}(\Omega)$, such that (\ref{def:PoincareInequality}) is an equality for $C=\poinc(\mu)$, we say that the inequality is \emph{saturated} by $f_{\mathrm{opt}}$.
\end{defi}
This definition means $\var_\mu(f) \leq C(\mu) \int_\Omega (f')^2 d\mu $ for weakly differentiable 
functions $f$ such that $f$ and $f'$ are square integrable for $\mu$. Hence Definition \ref{ded:poinc1} is formally stronger as it involves general weakly differentiable functions.
Nevertheless, the Poincar\'e inequality of  Definition \ref{ded:poinc2} implies the one of  Definition \ref{ded:poinc1}: first it is enough 
to consider functions $f$ with square integrable weak derivative $f'$, then one can apply the Poincar\'e inequality to certain truncations of $f$, which belong to $\Sob{1}{\mu}(\Omega)$, in order to  show that $f$ is also necessarily square integrable for $\mu$.

Poincar\'e inequalities are often stated in terms of the ratio of energies of a function and its derivative:
\begin{defi}
Let $f$ in $\Sob{1}{\mu}(\Omega)$ with $\int f^2 d\mu > 0$. The \emph{Rayleigh ratio} of $f$ is:
\begin{equation}
J(f) = \frac{\Vert f' \Vert^2}{\Vert f \Vert^2} = \frac{\int_\Omega f'^2 d\mu}{\int_\Omega f^2 d\mu}
\end{equation}
\end{defi}
Thus $\mu$ admits a Poincar\'e inequality if and only if the Rayleigh ratio admits a positive  lower bound over the subspace of centered functions $\SobCent{1}{\mu}(\Omega) = \{f \in \Sob{1}{\mu}(\Omega), \int_\Omega f d\mu = 0\}$. In that case,
$$\poinc(\mu) = \left( \inf_{f \in \SobCent{1}{\mu}(\Omega) - \{0\}} J(f) \right)^{-1}= \sup_{f \in \SobCent{1}{\mu}(\Omega) - \{0\}} J(f)^{-1}$$
The expression of $\poinc(\mu)$ as a supremum makes it easy to compute lower bounds, 
by choosing an appropriate test function $f$. An example is given below:
\begin{prop}
 $\poinc(\mu) \geq \var_\mu$, with equality if $\mu = \mathcal{N}(0,1).$
\end{prop}
\begin{proof}
 The lower bound is obtained for $f(x) = x - \E_\mu$. The equality case is well-known (see e.g. \citep{Ane_2000}).
\end{proof}


\subsection{A general criterion} 
The class of probability measures that admit a Poincar\'e inequality has been completely characterized in the work of Muckenhoupt \cite{Muckenhoupt_1972}. See also  \cite{Bobkov2009} for refinements. 

\begin{thm}[\cite{Muckenhoupt_1972}] \label{prop:Muckenhoupt}
 Let $\mu(dt) = \pdf(t)dt$ be a probability measure on  $\Omega=(a,b)$. Let 
 $m$ be a median of $\mu$  and define:
$$ A_-= \sup_{a < x <m}  \mu\big((a,x)\big) \int_x^m \frac{1}{\pdf(t)}dt, \qquad
A_+ = \sup_{m < x <b} \mu\big((x,b)\big) \int_m^x \frac{1}{\pdf(t)}dt,$$
with the convention $0\cdot\infty=0$.
Then  $\mu$ admits a Poincar\'e inequality iff $A_-$ and $ A_+$ are finite, and in this case
\begin{equation} \label{eq:Muckenhoupt}
\frac{1}{2} \max(A_-, A_+) \leq \poinc(\mu) \leq 4\max(A_-, A_+).
\end{equation}
\end{thm}
This result explains how to estimate, up to a multiplicative factor, the Poincar\'e constant by 
a simpler quantity. As explained in \cite{BCR}, it can be interpreted as a reduction
from functions to sets (by decomposition functions according to their level sets): a Poincar\'e 
inequality is equivalent to a comparison between the measures of sets, and a notion of $\mu$-capacity. In one dimension, a further reduction allows to restrict to increasing function vanishing
at $m$, for which level sets for positive values are of the form $(x,b)$.

\subsection{Perturbation and transport}
Given a measure verifying a Poincar\'e inequality, one may try to deform it into another measure still having a finite Poincar\'e constant. Many such results exist in the literature. Here we mention two fundamental ones: the bounded perturbation principle, and the Lipschitz transportation principle.
They hold in very general  settings, but for the purpose of this article, it is enough
to state them on $\R$.

\begin{lem} \label{lem:perturbationPrinciple}
Let $\mu$ be a probability measure on $\Omega\subset \R$. Let $\psi:\Omega \to \mathbb R$ be a bounded measurable function and let $\tilde{\mu}$
be the probability measure given by $d\tilde{\mu}=e^{\psi} d\mu/Z$ where $Z$ is the normalizing constant. Then 
$$ \poinc(\tilde{\mu}) \le e^{\sup \psi-\inf\psi} \poinc(\mu).$$
\end{lem}

This fact is easily proved at the level of Poincar\'e inequalities by using $\mathrm{Var}_{\tilde{\mu}}=\inf_{a\in \mathbb R}
\int (f-a)^2 d\tilde{\mu}$ and applying obvious bounds on $\psi$.

\begin{lem} \label{lem:lipshitzTransport}
Let $\mu$ be an absolutely continuous probability measure on an open interval  $\Omega\subset\mathbb R$.
Let $T:\Omega \to \R$ be a Lipschitz map, (i.e. verifying that there exists $L\in \mathbb R$ such that for all $x,y\in \Omega$,  $|T(x)-T(y)|\le L |x-y| $) . Assume that $T\mu$, the image measure of $\mu$ by $T$,  is also absolutely continuous. Then
$$\poinc(T\mu)\le \|T\|_{Lip}^2 \poinc(\mu),$$ 
 where  $ \|T\|_{Lip}$ is the smallest possible value of the constant $L$ above.
\end{lem}

\begin{proof}
It is convenient here to work with locally Lipschitz functions (i.e. functions which are Lipschitz on any compact interval). They are weakly differentiable with locally bounded derivative (see e.g.  \cite{Brezis}). However one can also define pointwise and as a whole their  ``absolute value of the derivative'' as 
$$|f'|(x):=\limsup_{y\to x} \frac{|f(y)-f(x)|}{|y-x|}.$$ 
By Rademacher's theorem, a locally Lipschitz function is differentiable almost everywhere, 
and the latter coincides a.e. with the absolute value of the derivative of $f'(x)$. 
By hypothesis, $|f'|(x)=|f'(x)|$ will hold almost surely in the sense of $\mu$ and $T\mu$ as well.

Set $\nu=T\mu$. By a density argument it is enough to prove the Poincar\'e inequality for $\nu$
for all locally Lipschitz functions.
Let $g$ be a locally Lipschitz function on $\mathbb R$. Then $f:=g\circ T$ is also locally Lipschitz on $\Omega$, and $\mathrm{Var}_\mu(f)\le \poinc(\mu) \int |f'|^2 d\mu$.
 Since $\nu$ is the image measure of $\mu$ by $T$, $\mathrm{Var}_\mu(f)=\mathrm{Var}_\nu(g)$. Observe that $f=g\circ T$ is also
 locally Lipschitz and verifies $|f'|(x)\le L |g'|(T(x))$ for all $x$. Consequently
   $$\int |f'|^2 d\mu\le L^2 \int  (|g'|(T(x)))^2 d\mu(x)=   L^2 \int  |g'|^2 d\nu.$$
Hence we have proved that $\mathrm{Var}_\nu(g) \le \poinc(\mu)  L^2 \int  |g'|^2 d\nu.$
The result follows.
\end{proof}

In the case of probability measures $\mu, \nu$ on the real line, and when  $\mu$ has no atoms,
a natural map which pushes $\mu$ forward to $\nu$ is the monotonic map $T:=F_\nu^{-1}\circ F_\mu$ (here $F_\nu^{-1}$ stands for the  
generalized left inverse). It remains to estimate the Lipschitz norm of $T$.

\begin{lem}
Let $\mu$ and $\nu$ be probability measures on $\R$. Assume that $\mu(dt)=\pdf_\mu(t) dt$ where $\pdf_\mu$ is positive and continous on $(a_\mu, b_\mu)$ ($-\infty\le a_\mu, b_\mu\le+\infty$) and   vanishes outside. Let us make the same structural assumption for $\nu$. Then $T:=F_\nu^{-1}\circ F_\mu$ is well defined and differentiable on $(a_\mu, b_\mu)$, with
 $ T'=\pdf_\mu/ \pdf_\nu\circ F_\nu^{-1}\circ F_\mu$. Consequently its Lipschitz norm is 
 $$\|T\|_{Lip}= \sup_{(a_\mu, b_\mu)} \frac{\pdf_\mu}{ \pdf_\nu\circ F_\nu^{-1}\circ F_\mu}=\sup_{(0,1)} \frac{\pdf_\mu\circ F_\mu^{-1} }{ \pdf_\nu\circ F_\nu^{-1}} =\sup_{(a_\nu, b_\nu)} \frac{\pdf_\mu\circ F_\mu^{-1}\circ F_\nu }{ \pdf_\nu}.$$
\end{lem}

The previous two lemmas can be applied when $\mu(dx)=\exp(-|x|) dx/2$, the double exponential measure. In this case $\pdf_\nu\circ F_\nu^{-1}(t)=\min(t,1-t)$ and 
$\poinc(\mu)=4$. This is how Bobkov and Houdr\'e \cite{Bobkov_Houdre_1997} established the following estimate: 
   
\begin{equation}\label{eq:BHtransport}
   \poinc(\nu)\le 4 \left(  \sup_{x\in (a_\nu, b_\nu)} \frac{\min(F_\nu(x), 1-F_\nu(x))}{ \pdf_\nu(x)}\right)^2.
   \end{equation}  
  These authors also deduced from this approach that for log-concave probability measures $\nu$  on $\R$, with median $m$, $\poinc(\nu)\le 1/\pdf_\nu(m)^2$.
  \smallskip
   
Actually,  we can improve on (\ref{eq:BHtransport}) by choosing the logistic measure $\mu(dx)=\frac{e^x}{(1+e^x)^2} dx$ instead of the double exponential measure in the proof. Indeed, in this case
$\pdf_\mu\circ F_\mu^{-1}(t)=t(1-t)$ is smaller than $\min(t,1-t)$, while the Poincar\'e constant is the same 
$\poinc(\mu)=4$, as shown in \cite{Barthe_logistic}. Consequently

\begin{cor}
Let $\nu$ be a probability measure on $ (a, b)\subset \R$, with a positive continous density on $(a,b)$. Then
  $$\poinc(\nu)\le 4 \left(  \sup_{x\in (a, b)} \frac{F_\nu(x)( 1-F_\nu(x))}{ \pdf_\nu(x)}\right)^2.$$
\end{cor}

\begin{rem}
We have mentioned in the introduction that rewriting the sensitivity analysis questions in terms of uniform variables does not provide good results for our purposes. This approach amounts to consider the monotone transport from uniform variables to the variables of interest, and 
this transport is not Lipschitz for unbounded variables. The results of the present section show that 
rewriting the problem in terms of logistic variables is more effective.
\end{rem}


\subsection{Spectral interpretation} \label{sec:spectral}
When $\mu$ has a continuous density that does not vanish on a compact interval $[a,b]$, then $\mu$ admits a Poincar\'e constant on $[a,b]$,  as follows from  the Muckenhoupt condition (Theorem \ref{prop:Muckenhoupt}) or from the bounded perturbation principle (Lemma~\ref{lem:perturbationPrinciple}).
Furthermore, modulo a small additional regularity condition, the Poincar\'e inequality is saturated and the Poincar\'e constant is related to a spectral problem, as detailed now.

\begin{thm} \label{prop:SpectralTheorem}
Let $\Omega=(a,b)$ be a bounded open interval of the real line, and assume that $\mu(dt)=\pdf(t)dt=e^{-V(t)}dt$ where $V$ is continuous and piecewise $C^1$ on $\overline{\Omega}=[a,b]$. Consider the three following problems:\\
\begin{itemize}
\item[(P1)] $\textrm{Find } f \in \Sob{1}{\mu}(\Omega) \textrm{ s.t. } \quad
J(f) = \frac{\Vert f' \Vert^2}{\Vert f \Vert^2} \quad \textrm{ is minimum under} \quad \int f d\mu=0$ \\
\item[(P2)] $\textrm{Find } f \in \Sob{1}{\mu}(\Omega) \textrm{ s.t. } \quad
\langle f', g' \rangle = \lambda \langle f, g \rangle \quad \forall g \in \Sob{1}{\mu}(\Omega)$ \\
\item[(P3)] $\textrm{Find } f \in \Sob{2}{\mu}(\Omega) \textrm{ s.t. } \quad f'' - V'f' = - \lambda f \quad \textrm{ and } \quad f'(a) = f'(b) = 0$ \\
\end{itemize}
Then the eigenvalue problems (P2) and (P3) are equivalent, and their eigenvalues form an increasing sequence $(\lambda_k)_{k \geq 0}$ of non-negative real numbers that tends to infinity. Moreover $0$ is an eigenvalue, and all eigenvalues are simple.
The eigenvectors $(u_k)_{k \geq 0}$ form a Hilbert basis of $L^2(\mu)$ and verify:
\begin{equation}
u_k'(x) = \frac{\lambda_k}{\pdf(x)} \int_x^b u_k(t)\pdf(t)dt
\end{equation}
In particular, up to a modification with Lebesgue measure zero, the $u_k$'s are $C^1$ on $[a,b]$ and if $\pdf$ is of class $C^k$, $u_k$ is of class $C^{k+1}$.\\
\noindent
Furthermore when $\lambda = \lambda_1$, the first positive eigenvalue, $(P2)$ and $(P3)$ are equivalent to $(P1)$ and the minimum of $(P1)$ is attained for $f = u_1$. In other words, the optimal Poincar\'e constant is $\poinc(\mu)= 1/\lambda_1$ and the inequality is saturated by a non-zero solution of (P3). Finally, $u_1$ is strictly monotonic.
\end{thm}

Although the equivalence between (P1) and (P3) is well-known, it was not easy to find a self-contained proof. Several arguments can be found in the literature of functional inequalities, and especially in \cite{Bobkov2009} for mixed Dirichlet and Neumann boundary conditions $f(a)=f'(b)=0$. We also found complementary arguments in the literature of numerical analysis, for instance in \cite{Dautray_Lions_vol3}, for the uniform distribution. 
Since neither of them are fully corresponding to our setting, we have inserted below a proof of Theorem~\ref{prop:SpectralTheorem}, making a synthesis of the ideas of the two literature fields. In particular, we adopt the Hilbertian point of view chosen in \cite{Dautray_Lions_vol3}, which is central for estimating the optimal Poincar\'e constant, as seen in Section~\ref{sec:numerical_method}, and adapt the short and elementary proofs of \cite{Bobkov2009} coming under more general principles (maximum principle, regularity theorem). 

\begin{proof}[Proof of Theorem~\ref{prop:SpectralTheorem}]
Since $\Omega$ is bounded an $\pdf$ is continuous on $\overline{\Omega}$, there exists two positive real numbers $m, M$ such that:
$$\forall t \in \overline{\Omega}, \quad 0 < m  \leq \pdf(t) \leq M$$
Then, as mentioned in \S~\ref{sec:Definitions}, $L^2(\mu)$ and all weighted Sobolev spaces $\Sob{\ell}{\mu}(\Omega)$ are topologically unchanged when we replace $\mu$ by the Lebesgue measure on $\Omega$. This allows using the spectral theory of elliptic problems (see e.g. \cite{Allaire_Book}, Chap. 7) on usual Sobolev spaces with bounded $\Omega$.\\

More precisely, let us consider problem (P2), and denote $\mathcal{H} = L^2(\mu)$ and $\mathcal{V} = \Sob{1}{\mu}(\Omega)$. Then the injection $\mathcal{V} \subset \mathcal{H}$ is compact, and $\mathcal{V}$ is dense in $\mathcal{H}$.
(P2) can be written as an eigenvalue problem of the form:
\begin{equation} \label{eq:weak_formulation_initial}
a(f, g) = \lambda \langle f, g \rangle \qquad \forall g \in  \mathcal{V}
\end{equation}
where $a(f, g) = \langle f', g' \rangle$ and $\langle ., . \rangle$ denotes the scalar product on $\mathcal{H}$. 
In order to apply spectral theory, $a$ should be coercive with respect to $\mathcal{V}$, i.e. $\exists C>0$ s.t. $a(f,f) \geq C \langle f, f \rangle_{\Sob{1}{\mu}(\Omega)}$ for all $f$, which is not true here. A convenient way to overcome this issue is to consider the equivalent shifted problem (see \cite{Dautray_Lions_vol3}, \S 8.2.): 
\begin{equation} \label{eq:weak_formulation}
\alpha(f, g) = (\lambda+1) \langle f, g \rangle \qquad \forall g \in \mathcal{V}
\end{equation}
where $\alpha(f,g) = \langle f', g' \rangle + \langle f, g \rangle$ 
is the usual scalar product on $\mathcal{V}$.
It is coercive on $\mathcal{V}$.
Then we can apply Theorem~7.3.2 in \cite{Allaire_Book}: The possible eigenvalues form an increasing sequence $(\lambda_k+1)_{k \geq 0}$ of non-negative values that tends to infinity
and the eigenvectors $(u_k)_{k \geq 0}$ form a Hilbert basis of $L^2(\mu)$. Further, from (\ref{eq:weak_formulation_initial}) with $f=g$, we have $\lambda_k \geq 0$.
Now remark that $0$ is an eigenvalue and the corresponding eigenspace is spanned by the constant function $u_0 = 1$. Indeed, taking $g=f$ in (\ref{def:PoincareInequality}) leads to $\int_a^b f'(x)^2 e^{-V(x)}dx =0$ and $f$ is a constant function.\\

Now let us prove the equivalence between (P2) and (P3). We restrict the presentation to  $\lambda>0$ since the case $\lambda=0$ is direct, using a Sturm-Liouville form of (P3), i.e. $(f'V)' = -\lambda f$. Formally the link comes from an integration by part of the left hand side of (P2):
\begin{equation} \label{eq:IntByPart}
\int_a^b f' g' \pdf= f'(b)\pdf(b)g(b) - f'(a)\pdf(a)g(a) - \int_a^b (f'\pdf)' g
\end{equation}
Since this must be equal to $\lambda \int_a^b fg\pdf$ for all $g \in \Sob{1}{\mu}(\Omega)$, we should have $f'(a)=f'(b)=0$ and $(f'\pdf)' = - \lambda f \pdf$ which is problem (P3).
Actually this method, read in the reverse sense, shows that (P3) implies (P2).
Indeed if $f \in \Sob{2}{\mu}(\Omega)$ then $f'\pdf \in \Sob{1}{\mu}(\Omega)$ and the integration by part is valid.
However, to see that (P2) implies (P3), an argument of regularity must be used since $f$ is only assumed to belong to $\Sob{1}{\mu}(\Omega)$.
This is achieved by proving that, $\mu$-almost surely,
\begin{equation} \label{eq:IntegralFormOfPrime}
f'(x) = \frac{\lambda}{\pdf(x)} \int_x^b f(t)\pdf(t)dt.
\end{equation}
Indeed, this implies that $f'\pdf$ is $C^1$ on $[a,b]$
and (\ref{eq:IntByPart}) is valid. Furthermore, (\ref{eq:IntegralFormOfPrime}) also implies $f'(a)=f'(b)=0$ and this gives (P3).
The proof of (\ref{eq:IntegralFormOfPrime}) is postponed to the appendix.\\

Now let us prove that all eigenvalues are simple. Notice that, from (\ref{eq:IntegralFormOfPrime}), all solutions of (P3) are locally Lipschitz, up to a modification of Lebesgue measure $0$. Thus when $\lambda$ is an eigenvalue, the vector space of solutions of the second order linear differential equation $f'' - V'f' = -\lambda f$ has dimension 2. Adding Neumann conditions, the dimension is equal to 1.
Indeed, let $f$ and $g$ be two solutions of $(P3)$. Then $f(b) \neq 0$ otherwise $f$ would be identically null by uniqueness of the Cauchy problem (since $f'(b)=0$). Hence there exists a real number $c$ such that $g(b) = c f(b)$. Then consider $h = g - cf$. We have $h(b)=h'(b)=0$ implying, by unicity of the Cauchy problem, that $h=0$, i.e. $g = c f$.\\

In this Hilbertian framework, the equivalence between (P2) and (P1) is easily seen. Indeed, if $f \in \Sob{1}{\mu}(\Omega)$ verifies $\int f d\mu  = 0$ then $f$ is orthogonal to $1$ and hence $f = \sum_{k = 1}^{+\infty} f_k u_k$. Then, it holds:
$$ J(f) = \frac{\Vert f' \Vert^2}{\Vert f \Vert^2} 
= \frac{\alpha(f,f)}{\Vert f \Vert^2} - 1 
= \frac{\sum_{k = 1}^{+\infty} \lambda_k f_k^2}{\sum_{k = 1}^{+\infty}f_k^2} \geq \lambda_1$$
with equality iff $f$ is proportional to $u_1$. This also proves that the minimum in (P1) is attained precisely on the 1-dimensional space spanned by $u_1$.\\

Finally we report to the Appendix the proof of monotonicity properties as well as connection to variation calculus.
\end{proof}

\begin{rem} Another solution to make coercive the bilinear form $a(f,g)=\langle f',g' \rangle$  is to project onto the space of centered functions $\SobCent{1}{\mu}(\Omega)$, as done in \cite{Dautray_Lions_vol3}. Indeed, on that space, coercivity is equivalent to the existence of a Poincar\'e constant, which can be proved independently by the bounded perturbation principle as mentioned above. However, this seems a less convenient setting for the numerical method developed in Section~\ref{sec:numerical_method} since the `hat' functions used do not belong to $\SobCent{1}{\mu}(\Omega)$.
\end{rem}

\begin{defi}[Spectral gap] \label{def:sg}
The smallest strictly positive eigenvalue $\lambda_1$ of $Lf = f'' - V'f'$ with Neumann boundary conditions $f'(a)=f'(b)=0$ is called \emph{Neumann spectral gap}, or simply \emph{spectral gap}, and denoted by $\sg(\mu)$.
\end{defi}

We conclude this section by a proposition showing that, under regularity conditions on $V$, the Neumann problem (P3) can be written as a Dirichlet problem.
\begin{prop} \label{prop:NeumannEquivDirichlet}
If $V$ is of class $C^2$, then (P3) is equivalent to:
\begin{itemize}
\item[(P4)] $\textrm{Find } h \in \Sob{2}{\mu}(\Omega) \textrm{ s.t. } \quad h'' - V'h' - V''h = - \lambda h \quad \textrm{ and } \quad h(a) = h(b) = 0$ \\
\end{itemize}
\end{prop}

\begin{proof} 
Define $L$ and $K$ the operators
$$Lf = f'' - V'f' \qquad Kf = f''-V'f' -V''f$$
Let $f$ be a solution of (P3). Since $V$ is of class $C^2$ then by Theorem~\ref{prop:SpectralTheorem}, $f \in \mathcal{H}^3(\mu)$. Now define $h=f'$. We have:
$K(f')=(Lf)'=-\lambda f'$ and $h$ is solution of (P4).\\
Conversely, let $h$ be a solution of (P4) and let $f$ be the primitive function of $h$ such that $\int f d\mu = 0$. Then $(Lf)' = -\lambda f'$ and there exists $c \in \mathbb{R}$ such that $Lf = -\lambda f + c$. Remarking that $\int_a^b L(f) e^{-V} = \int_a^b (f'e^{-V})' = 0$, the centering condition implies $c=0$ which proves that $f$ is a solution of (P3).
\end{proof}

\subsection{More techniques}
We conclude this quick overview by briefly quoting two other methods.

The first one plays a major role in the study of multidimensional functional inequalities. It has its roots in early works by H\"ormander and Brascamp-Lieb, and was developped by Bakry and his collaborators who introduced and exploited a notion of curvature and dimension for diffusion generators, see e.g. \cite{Ane_2000}. The simplest instance of these results is as follows:

\begin{thm}
Let $\mu$ be a probability measure on an open interval $\Omega$ with $d\mu(x)=e^{-V(x)} 1_\Omega(x) \,dx$ where $V$ is twice continuously differentiable. If there exists $r>0$ such that  for all $x\in \Omega$, $V''(x)\ge r$ then $\poinc(\mu)\le 1/r$.
\end{thm}
The estimate is exact for the standard Gaussian measure.

\medskip
The second result is known as Chen's variational formula \cite{Chen}. Its main advantage is to express the Poincar\'e constant as an infimum  (while, as already mentioned, its definition is given as a supremum). This allows to obtain upper bounds rather easily. We state it in a different form, which appeared in \cite{DW}. 

\begin{thm}
Let $d\mu(x)=e^{-V(x)} 1_{[a,b]}(x) dx$ be a probability measure on a segment, with $V$ twice continuously differentiable on $[a,b]$. Then 
$$\sg(\mu)=\frac{1}{\poinc(\mu)}=\sup_{g'>0}\, \inf_{(a,b)} \frac{(-Lg)'}{g'},$$
where $Lf=f''-V'f'$ and the supremum is taken over three times differentiable functions on $(a,b)$, having positive first derivative.
\end{thm}

As an example of application, Chen's formula recovers the following classical statement, which often allows to calculate the spectral gap. It is the reciprocal of the fact that a solution of the spectral problem (P3) corresponding to the smallest positive eigenvalue is necessarily strictly monotonic (one part of Theorem~\ref{prop:SpectralTheorem}). 
\begin{cor} \label{cor:sgWhenMonotoneSolution}
Under the framework and notations of Section~\ref{sec:spectral}, if $f$ is a strictly monotonic solution of problem (P3) 
$$Lf = -\lambda f, \quad f'(a) = f'(b) = 0$$ 
for some $\lambda>0$, then $\lambda$ is the spectral gap.\\
The same is true for problem (P4) replacing `monotonic' by `non vanishing'.
\end{cor}
\begin{proof}
By definition of the spectral gap, $\sg(\mu) \leq \lambda$. Conversely, up to a change of sign we can assume that $f'>0$ on $(a,b)$. Then Chen's formula immediately gives $\sg(\mu) \geq \inf_{(a,b)} \frac{(-Lf)'}{f'} = \lambda$.\\
The last part comes from the equivalence between (P3) and (P4): If $h$ is a non vanishing solution of (P4), then the primitive function of $h$ such that $\int f d\mu = 0$ is a strictly monotonic solution of (P3) (see the proof of Prop.~\ref{prop:NeumannEquivDirichlet}).
\end{proof}

\section{Poincar\'e inequalities on intervals}  \label{sec:generalResults}

This section gathers tools for the study of the Poincar\'e constants of the restrictions of a given measures to subintervals.

\subsection{General results}
The Poincar\'e constant cannot increase by restriction to a subinterval:
\begin{lem}\label{lem:restriction}
Let $\mu$ be a probability measure on $\R$, of the form $d\mu(t)=1_\Omega(t) \pdf(t)\, dt$
where $\Omega$ is an open interval on which $\pdf$ is continuous and positive. Let $I\subset \Omega$ be an open interval, and $\mu_{|I}$ denote the probability measure obtained by restriction of $\mu$ to $I$ (defined by $\mu_{|I}(A)=\mu(I\cap A)/\mu(I)$). Then
$$\poinc(\mu_{|I})\le \poinc(\mu).$$
Consequently, the map defined on intervals by $I\mapsto \poinc(\mu_{|I})$ is non-decreasing.
\end{lem}
\begin{proof}
Let $f\in \mathcal H^1(\mu_{|I})$. Set $\Omega=(\alpha, \beta)$ and $I=(\alpha',\beta')$. If $\alpha<\alpha'$  then by  hypothesis, $\pdf$ is bounded from above and below in a neighborhood of $\alpha'$, which ensures that $f$ has a limit at $\alpha'$, and can be continuously extended by a constant on $(\alpha,\alpha']$. The same applies if $\beta'<\beta$. This allows to build a weakly differentiable function $\tilde{f}$ on $\Omega$ which concides with $f$ on $I$ and has zero derivative outside $I$. The result easily follows:
\begin{eqnarray*}
\var_{\mu_{|I}}(f)&=&\inf_a \int_I (f-a)^2 \frac{d\mu}{\mu(I)} 
\le  \inf_a \int_\Omega (\tilde f-a)^2 \frac{d\mu}{\mu(I)} \\
&\le & \poinc(\mu) \int_\Omega (\tilde f')^2 \frac{d\mu}{\mu(I)}
= \poinc(\mu) \int_I ( f')^2 \frac{d\mu}{\mu(I)}\\
&=&\poinc(\mu) \int_I ( f')^2 d\mu_{|I}.
\end{eqnarray*}
\end{proof}

The previous result helps proving a continuity type property with respect to the support:
\begin{lem} \label{lem:continuitySupport}
Consider again the framework of Lemma \ref{lem:restriction}, and let $I_\epsilon$ be a family and increasing subintervals such that $I_\epsilon \uparrow \Omega$ when $\epsilon \rightarrow 0$. Typically, $I_\epsilon = (a + \epsilon, b- \epsilon)$ for $\Omega = (a,b)$, $I_\epsilon = (a, a + 1/\epsilon)$ for $\Omega = (a, +\infty)$ and so on.
Then the Poincar\'e constant on $\Omega$ is the limit of the Poincar\'e constant on subintervals:
$$ \poinc(\mu) = \underset{\epsilon \rightarrow 0}{\lim} \, \poinc(\mu_{\vert I_\epsilon})$$
\end{lem}
\begin{proof}
Firstly, by Lemma~\ref{lem:restriction}, the sequence $\poinc(\mu_{\vert I_\epsilon})$ is increasing and bounded by  $\poinc(\mu)$. 
Hence $ \poinc(\mu) \geq \underset{\epsilon > 0}{\sup} \, \poinc(\mu_{\vert I_\epsilon}) = \underset{\epsilon \rightarrow 0}{\lim} \, \poinc(\mu_{\vert I_\epsilon})$.\\
Conversely, let $f \in \mathcal{H}^1_{\mu}(\Omega)$.
Observe that, by Lebesgue theorem, 
$$\mu(I_\epsilon) = \int_{I_\epsilon} \pdf(t)dt \underset{\epsilon \rightarrow 0}{\rightarrow} \mu(\Omega) = 1$$
and for all $\mu$-integrable function $g$:
$$\mathbb E_\mu(g) = \int_\Omega g(t)\pdf(t)dt 
= \underset{\epsilon \rightarrow 0}{\lim} \, \int_{I_\epsilon} g(t)\pdf(t)dt 
= \underset{\epsilon \rightarrow 0}{\lim} \, \int_{I_\epsilon} g(t)\frac{\pdf(t)}{\mu(I_\epsilon)}dt 
= \underset{\epsilon \rightarrow 0}{\lim} \, \mathbb E_{\mu_{\vert I_\epsilon}}(g). $$
Applying this result to $f$, $f^2$ and $f'^2$, it holds:
$$\frac{\var_\mu (f)}{\int_\Omega f'^2 d\mu} 
= \underset{\epsilon \rightarrow 0}{\lim} \, \frac{\var_{\mu_{\vert I_\epsilon}} (f)}{\int_{I_\epsilon} f'^2 d \mu_{\vert I_\epsilon}} 
\leq \underset{\epsilon \rightarrow 0}{\lim} \, \poinc(\mu_{\vert I_\epsilon})
$$
Since it is true for all $f \in \mathcal{H}^1_{\mu}(\Omega)$, we get 
$ \poinc(\mu) \leq \underset{\epsilon \rightarrow 0}{\lim} \, \poinc(\mu_{\vert I_\epsilon})$
\end{proof}


\subsection{Improvements for symmetric settings}
 In general, the Poincar\'e constant of a restriction  to an interval is at most the one of the initial measure. A better result is available in symmetric settings:
\begin{lem} \label{lem:TransportSymmetricInterval}
 Let $\mu$ be a probability measure on $\R$. Assume $\mu(dx)=\pdf_\mu(x)dx$ where $\pdf_\mu$ is a unimodal function with a maximum at $x=0$.
 Let $I\subset \R$ be a non-empty interval and $\nu:= \mu(\cdot\cap I)/\mu(I)$. Assume in addition that $\mu(\R^-)=\nu(\R^-)$.
 Then the monotonic transport from $\mu$ to $\nu$ is $\mu(I)$-Lipschitz. Consequently
   $$\poinc(\nu)\le \mu(I)^2 \poinc(\mu).$$
\end{lem}   
\begin{proof}
By hypothesis, $F_\nu(0)=F_\mu(0)$. Hence $T(0)=0$ and since $T$ is increasing, $T(x) $ and $x$ have the same sign.

We claim that for all $x$ (in the interior of the support of $\mu$), $|T(x)| \le |x|$. Indeed
for $x\ge 0$, this inequality amount to $F_\mu(x)\le F_\nu(x)$. The latter is obvious 
when $x> \sup(I)$ since in this case $F_\nu(x)=1$. For $0\le x\le \sup(I)$, the claim boils
down to 
 $$\mu(\R^-)+\int_0^x \pdf_\mu \le \nu(\R^-)+\int_0^x \frac{\pdf_\mu}{\mu(I)},$$
 which is obvious using $\mu(\R^-)=\nu(\R^-)$.

For $x\le 0$, the claim amounts to $x\le T(x)$, that is $F_\nu(x)\le F_\mu(x)$. This is obvious
if $x<\inf(I)$. If $\inf(I)\le x\le 0$, the inequality amounts to
  $$\nu(\R^-)-\int_x^0 \frac{\pdf_\mu}{\mu(I)} \le \nu(\R^-)-\int_x^0 \pdf_\mu,$$
  which is easily veryfied.

Eventually, for $x$ in the interior of the support of $\mu$, 
$$T'(x)=\frac{\pdf_\mu(x)}{\pdf_\nu(T(x))}=\mu(I) \frac{\pdf_\mu(x)}{\pdf_\mu(T(x))} \le \mu(I),$$
since on $\R^+$, $0\le T(x) \le x$ and $\pdf_\mu$ is non-increasing, and on
$\R^-$, $x \le T(x)\le 0$ and $\pdf_\mu$ is non-decreasing.
\end{proof}


Our goal here is to get finer results when the underlying measure is even. A first natural observation 
is that extremal (or almost extremal) functions for the Poincar\'e inequality can be found among odd functions. 

\begin{lem}\label{lem:symmetry-quotient}
Let $\mu$ be an even probability measure on $\mathbb R$. Assume that it is supported on a symmetric interval, and is absolutely continuous, with a positive continuous density in the interior of this interval.
Then for any non constant function $f\in \mathcal H^1(\mu)$, there exists an odd  function
$g \in \mathcal H^1(\mu)$ with 
$$\frac{\int (f')^2 d\mu}{\var_\mu(f)} \ge \frac{\int (g')^2 d\mu}{\var_\mu(g)}\cdot$$ 
\end{lem}
\begin{proof}
Let $(-b,b)$ denote the interior of the support of $\mu$. By hypothesis, $f(0)$ is well defined and by the variational definition of the variance 
$$\frac{\int (f')^2 d\mu}{\var_\mu(f)} \ge \frac{\int (f')^2 d\mu}{\int (f-f(0))^2d\mu}
= \frac{\int_{-b}^0 (f')^2 d\mu+\int_0^b (f')^2 d\mu}{\int_{-b}^0 (f-f(0))^2d\mu+\int_0^b (f-f(0))^2d\mu }.$$
Using the inequality $\frac{u+v}{x+y}\ge \min(\frac{u}{x},\frac{v}{y})$, which is valid for non-negative numbers $u,v,x,y$ such that $x+y>0$ (and with the convention that $v/0=+\infty$),
we get that 
$$\frac{\int (f')^2 d\mu}{\var_\mu(f)} \ge \min \left(\frac{\int_{-b}^0 (f')^2 d\mu}{\int_{-b}^0 (f-f(0))^2d\mu},
\frac{\int_0^b (f')^2 d\mu}{\int_0^b (f-f(0))^2d\mu }  \right).$$
In order to conclude, we consider the odd functions $g,h$ defined on $(-b,b)$, such that for all $x\in[0,b)$, $g(x)=f(x)-f(0)$, and for all $x\in (-b,0]$, $h(x)= f(x)-f(0)$. By symmetry of $\mu$,
 $$\frac{\int_{-b}^0 (f')^2 d\mu}{\int_{-b}^0 (f-f(0))^2d\mu}= \frac{\int (h')^2 d\mu}{\var_\mu(h)}, \quad \mathrm{and} \quad \frac{\int_0^b (f')^2 d\mu}{\int_0^b (f-f(0))^2d\mu}=\frac{\int (g')^2 d\mu}{\var_\mu(g)}  \cdot$$  
\end{proof}

Next, we derive from the above proposition a perturbation principle for the Neumann spectral gap. 
It should be compared with the  "bounded perturbation principle" that we recalled in the previous section.

\begin{prop} \label{prop:Neumann-perturb-paire}
Let $d\mu(x)=\mathbf 1_{(-b,b)} e^{-V(x)}dx$ be an even probability measure, with continuous potential $V$. Let $\varphi:(-b,b)\to \mathbb R^+$ be an even function. Assume that $\varphi$ is non-increasing on $[0,b)$ and that $\tilde{\mu}$
defined by $d\tilde{\mu}=\varphi \, d\mu$ is a probability measure. Then
$$ \poinc(\tilde{\mu}) \le \poinc(\mu).$$
\end{prop}

\begin{proof}
By the previous lemma, it is enough to establish a Poincar\'e inequality for $\tilde \mu$ for odd functions only. Let $f$ be an odd weakly differentiable function.  By Lemma \ref{lem:restriction},
for any $a\in (0,b)$, $\poinc(\mu_{|(-a,a)})\le \poinc(\mu)$. Since $f$ is odd and $\mu$ is even,
$\int_{(-a,a)} f\,d\mu=0$, and the inequality $\var_{\mu_{|(-a,a)}}(f)\le \poinc(\mu) \int (f')^2d\mu_{|(-a,a)}$ can be rewritten as 
$$\int f^2 1_{(-a,a)}d\mu \le \poinc(\mu) \int (f')^2 1_{(-a,a)} d\mu.$$
Next, for all $x\in (-b,b)$,
$$\varphi(x)=\int_0^{+\infty} 1_{t\le \varphi(x)} \,dt= \int_0^{+\infty} 1_{I_t}(x) \, dt,$$ 
where $I_t=\varphi^{-1}\big([t,+\infty)\big)$. Our hypotheses on $\varphi$ ensure that $I_t$ 
is empty, or is a symmetric interval included in $(-b,b)$. Hence, if it is not empty, it is equal, up to a negligible set to $(-a,a)$ for some $a$. Therefore we know that 
$$\int f^2 1_{I_t}d\mu \le \poinc(\mu) \int (f')^2 1_{I_t} d\mu,$$
is true for all $t\ge 0$. 
But, by definition, pointwise, $\varphi=\int_{\mathbb R^+} 1_{I_t}$. Thus integrating the latter inequality over $t$ yields 
$$\int f^2 \varphi \, d\mu \le \poinc(\mu) \int (f')^2\varphi \,  d\mu.$$
This can be rewritten as $\var_{\tilde \mu} f\le \poinc(\mu) \int (f')^2 d\tilde\mu.$
\end{proof}

As an example of application, using $V \equiv 0$ and $\phi(x)=e^{-x^2/2}$, one can bound the Poincar\'e constant of the Gaussian measure restricted to $[-b,b]$ by the one of the uniform measure on that interval:
$\poinc \left( \mathcal{N}(0,1)\vert [-b, b] \right) \leq \frac{4b^2}{\pi^2}.$

\section{Exact values of Poincar\'e constants} \label{sec:optimalConstants}
In this section, we show how to derive a Poincar\'e constant on an interval $\Omega$ either as a first zero of a function expressed analytically (semi-analytical approach), or by a full numerical method based on finite elements. We consider the setting of the spectral theorem (Thm. \ref{prop:SpectralTheorem}): $\Omega$ bounded and $\pdf > 0$ on $\overline{\Omega}$ (with $\pdf$ continuous on $\overline{\Omega}$ and piecewise $C^1$ on $\Omega$). The general case can be derived from it by the continuity type property of Lemma~\ref{lem:continuitySupport}. Hence, if $\Omega$ is unbounded or $\pdf$ vanishes at the boundary of $\Omega$, then the Poincar\'e constant is the limit of Poincar\'e constants on bounded subintervals on which $\pdf > 0$.

\subsection{Semi-analytical results}

This section presents some examples where the spectral gap is obtained in a closed-form expression or by numerically searching a zero of a function expressed analytically.
They are summarized in Table~\ref{table:exact_optimal_constant}.
The most famous one is about the uniform distribution on $[a, b]$, corresponding to a spectral problem for the Laplacian operator, $f'' = - \lambda f$. 
Another well-known result is about the truncated normal, for the intervals formed by consecutive zeros of Hermite polynomials. After recalling this result, we show how to extend it for a general interval 
and develop the case of truncated (double) exponential and triangular distributions.
Without loss of generality, we will consider only one value for the parameters, 
remarking that the general result is obtained by rescaling (Lemma \ref{lem:lipshitzTransport} applied to $T$ and $T^{-1}$ when $T$ is an affine function).\\

\begin{table}[h!] 
\begin{center}
\begin{tabular}{|c|c|c|c|}
\hline 
\textbf{Prob. Dist. $\mu$} & \textbf{Support} & \textbf{$\poinc(\mu)$} & \textbf{Form of $f_{\textrm{opt}}(x)$} \\ 
\hline 
Uniform & $\left[ -\frac12, \frac12 \right]$ & $1 / \pi^2$ & $\sin (\pi x)$ \\ 
\hline 
$\mathcal{N}(0, 1)$ & $\mathbb{R}$ & 1 & $x$ \\ 
\hline 
 & $[r_{n,i}, r_{n, i+1}]$ & $1/{(n+1)}$ & $H_{n+1}(x)$ \\ 
\hline
& \emph{$[a,b]$} & see Prop.~\ref{prop:SpectralGapTruncatedNormal} & 
related to Kummer's functions \\
\hline
Double exp.& $\mathbb{R}$ & 4 & $\times$ \\ 
\hline 
 & \emph{$[a,b] \subseteq \mathbbm{R}^+$}& \emph{$\left( \frac14 + \left(\frac \pi {b-a} \right)^2 \right)^{-1}$} & 
$e^{x/2}\cos \left( \frac \pi {b-a}  x + \phi \right)$ \\ \hline
 & General case &  See Prop.~\ref{prop:trunc_exp} &  \\ \hline
Triangular & \emph{$[-1,1]$} & \emph{$r_1^{-2} \approx 0.1729$} & 
$\textrm{sign}(x) J_0 \left( r_1 (1 - \vert x \vert) \right)$ \\ 
\hline 
\end{tabular}
\end{center}
\caption{Examples of analytical results of the spectral problem with Neumann conditions. 
We have used the following notations:
$H_n$ is the Hermite polynomial of degree $n$ and $r_{n,i}$ its $i^{\textrm{th}}$ zero,
$J_0$ is the Bessel function of the first kind with $\nu=0$, and $r_1$ is its first root.}
\label{table:exact_optimal_constant}
\end{table}

In all these examples, solutions of the second-order differential equation can be expressed analytically 
as a linear combination of odd and even known functions. The whole spectrum then corresponds to the zeros of one determinant obtained with boundary conditions, as detailed in the following proposition.

\begin{prop} \label{prop:exact_spectral_gap}
Let $f_\lambda, g_\lambda$ a basis of solutions of the second-order differential equation $u'' - V' u' = -\lambda u$. 
Then the spectral gap on $[a, b]$ is the first zero of 
$$ d(\lambda) = \det 
\begin{pmatrix} f'_\lambda(a) & g'_\lambda(a) \\ f'_\lambda(b) & g'_\lambda(b)\end{pmatrix}
= f'_\lambda(a)g'_\lambda(b) - f'_\lambda(b)g'_\lambda(a)$$
Furthermore if $V$ is even and  $a = -b$, the spectral gap is the first zero of $\lambda \mapsto g'_\lambda(b)$.
\end{prop}

\begin{rem} \label{rem:prop:exact_spectral_gap}
This result is immediately adapted to the equivalent spectral problem with Dirichlet conditions (P4), 
$u'' - V'u' - V''u = - \lambda u$.
In that case, with similar notations, the spectral gap is the first zero of 
$$ d(\lambda) = \det 
\begin{pmatrix} f_\lambda(a) & g_\lambda(a) \\ f_\lambda(b) & g_\lambda(b)\end{pmatrix}
= f_\lambda(a)g_\lambda(b) - f_\lambda(b)g_\lambda(a)$$
and if $a = -b$ with even $V$, it is the first zero of $\lambda \mapsto f_\lambda(b)$.
\end{rem}

\begin{proof}[Proof of Proposition~\ref{prop:exact_spectral_gap}]
If the interval is symmetric and $V$ even, then by Lemma~\ref{lem:symmetry-quotient} a solution can be found among odd functions, and we can assume that $u = g_\lambda$. The Neumann conditions are then equivalent to $g'_\lambda(b) = 0$, which proves the result.\\
In the general case, let $u$ be a non-zero solution of the spectral problem. Then $u$ can be  written  as $u = A f + B g$, for some $A,B\in \mathbb R$. 
The boundary conditions lead to the linear system $Xc = 0$, 
with $c = \begin{pmatrix} A \\ B \end{pmatrix}$ and 
$ X = \begin{pmatrix} f'_\lambda(a) & g'_\lambda(a) \\ f'_\lambda(b) & g'_\lambda(b)\end{pmatrix}$.
Now there exists a non zero solution if and only if the system is degenerated, i.e. $d(\lambda) = 0$.
This gives the whole spectrum and, in particular, the spectral gap is the first zero.
\end{proof}

\subsection*{Truncated exponential distribution}

\begin{prop} \label{prop:trunc_exp}
Let $\mu$ be the (double) exponential distribution on the real line, $\mu (dt) = \frac12 e^{- \vert t \vert} dt$, and let $a < b$.
The Poincar\'e constant of the restriction of $\mu$ to $[a,b]$ is 
$\poinc(\mu_{\vert [a,b]}) =  \left(\frac 14 + \omega^2 \right)^{-1}$, where:
\begin{itemize}
\item If $a, b$ have the same sign, $\omega = \frac \pi {\vert b-a \vert}$.
The inequality is saturated by $f(x) = e^{\vert x \vert /2} \left( -\omega \cos(\omega(x-a)) \right.$ 
$+ \left. \frac12 \sin(\omega (x-a) \right)$
\item If $a < 0 < b$, then $\omega > \frac \pi {b-a}$. More precisely $\omega$ is the zero of 
$x \mapsto \cotan(\vert a \vert x) + \cotan(\vert b \vert x) + 1/x$
in $]0, \min(\pi / \vert a \vert, \pi / \vert b \vert)[$.
The inequality is saturated by 
$$f(x) = e^{\vert x \vert /2} \left( A \cos(\omega x)  
+ \left(B - \frac{\sgn(x)+1}{2} \frac A \omega \right) \sin(\omega x) \right)$$ with 
$A = \omega \cotan(\omega a) - \frac12$ and $B = \frac12 \cotan(\omega a ) + \omega$.\\
In particular if $a = -b$ then $\omega$ is the zero of $x \mapsto 2\cotan(b x)+1/x$ in $]0, \pi / b [$,
and the inequality is saturated by $f(x) = e^{\vert x \vert /2} \sin(\omega x)$.
\end{itemize}
\end{prop}


\subsection*{Triangular distribution}

\begin{prop} \label{prop:triangular}
Let us consider the triangular distribution $\mathcal{T}$ supported by $[-1,1]$, 
$\mu (dt) = (1 - \vert t \vert) \mathbbm{1}_{[-1,1]}(t) dt$. 
Denote by $J_0$ the Bessel function of the first kind with parameter $\nu=0$ (\cite{Abra_Steg}, Chapter 9),
and $r_1$ its first root. 
Then the Poincar\'e constant for $\mathcal{T}$ is $r_1^{-2}$, approximately equal to $0.1729$,
saturated by $f(x) = \textrm{sign}(x) J_0 \left( r_1 (1 - \vert x \vert) \right)$.
\end{prop}

\begin{rem} The result, involving Bessel functions, shows a connection with the spectrum of the Laplacian on the unit disk $\mathcal{D}$. Actually, using the symmetry of the problem, we can see that the spectral problem for the triangular distribution is the same than for the Laplacian on $\mathcal{D}$ restricted to radial solutions, as investigated e.g. in \cite{Dautray_Lions_vol3}, \S~8.1.1d. Details are given in the appendix.
\end{rem}


\subsection*{Truncated normal distribution} 
Let us consider the standard Gaussian distribution 
$d\gamma(x) = \frac{1}{\sqrt{2\pi}} e^{-V(x)}$ with $V(x) = \frac{x^2}{2}$. 
The spectral problem (P3) associated to Poincar\'e inequality is relative to the operator $Lf = f'' - V'f'$: 
\begin{eqnarray*}
Lf &=& f'' - x f'
\end{eqnarray*}
The spectral gap is known on every interval formed by successive roots of Hermite polynomials (see e.g. \cite{Dautray_Lions_vol3}), as detailed in Prop.~\ref{prop:SpectralGapHermite}. 
Recall that Hermite polynomials $H_n$ are the orthonormal basis of polynomials 
for $L^2(d\gamma)$ with unitary highest coefficient:
$ \int_{-\infty}^\infty H_n(x)H_m(x)d\gamma(x) = \delta_{n,m}$.
The first ones are:
$$ H_0(x) = 1 \quad H_1(x) = x \quad H_2(x) = x^2 - 1 \quad H_3(x) = x^3 - 3x $$
and the following can be obtained by the recursive relation: $ H_{n+1}(x) = x H_n(x) - H'_n(x)$.
They are also shown to verify the identity $H'_{n+1}(x) = (n+1) H_n(x)$ which proves, 
together with the recursion formula, that $H_n$ is solution of the differential equation
$ Lf = - nf$. 

\begin{prop}[Spectral gap of the normal distribution, specific case]\label{prop:SpectralGapHermite}
Let $n\ge 1$, $x_1<\cdots<x_n$ the roots of the $n$-th Hermite polynomial $H_n$. 
Let $I$ be one of the intervals: $(-\infty,x_1]$, $[x_1,x_2]$,\ldots ,$[x_{n-1},x_n]$, $[x_n,+\infty)$. Then the (Neumann) spectral gap on $I$ is $n+1$.
\end{prop}
\begin{proof}
On each $I$, the function $H_n$ is of constant sign and vanishes on the boundary. 
Since $H_{n+1}'=(n+1) H_n$ then inside $I$, $H_{n+1}$ is strictly monotonic and $H_{n+1}'=0$ on the boundary. Furthermore, $H_{n+1}$ is solution of $Lf = -(n+1)f$.
Hence, by Corollary~\ref{cor:sgWhenMonotoneSolution}, $n+1$ is the spectral gap.
\end{proof}

In  the general case, the spectral gap of the truncated normal distribution is related to  hypergeometric series, and more specifically to \textit{confluent hypergeometric series or Kummer's function} (\cite{Abra_Steg}, Section 13).
The Kummer's function $M_{a_1, b_1}(z) = {}_1F_1 \left(a_1; b_1; z\right)$ 
is a series $\sum_{p \geq 0} x_p$ with $x_0=1$ satisfying 
\begin{equation}
\label{hypergeom_form}
\frac{x_{p+1}}{x_p} = \frac{(p+a_1) z}{(p+b_1)(p+1)}
\end{equation}

\begin{prop}[Spectral gap of the truncated normal distribution]\label{prop:SpectralGapTruncatedNormal}
Let $a < b$, and define:
$$h_0(\lambda,t) =  M_{\frac{1-\lambda}2; \frac12}\left(\frac{t^2}2\right)
\qquad h_1(\lambda,t) = M_{\frac{2-\lambda}2; \frac32}\left(\frac{t^2}2\right)$$ 
Notice that $h_0$ and $h_1$ generalize Hermite polynomials: 
When $\lambda$ is an odd (resp. even) positive integer, 
then $t \mapsto h_0(\lambda, t)$ (resp. $t \mapsto t h_1(\lambda, t)$)
is equal, up to a multiplicative constant, to the Hermite polynomial of degree $\lambda-1$;
In particular they have the same zeros 
(\cite{Abra_Steg}, formulas 13.6.17 \& 13.6.18).\\
Denote $\truncnorm$ the standard normal distribution $\mathcal{N}(0,1)$ truncated on $[a, b]$. Then its spectral gap $\sg{(\truncnorm)}$ is the first zero of the function
$$ d(.) = b h_0(., a) h_1(., b) - a h_0(., b) h_1(., a)$$
Furthermore:
\begin{itemize}
\item If there exists $\lambda$ such that $a, b$ are two successive zeros of $h_0(\lambda, .)$ or two successive zeros of $h_1(\lambda, .)$ then $\sg(\truncnorm) = \lambda$
\item If $a= -b$ then $\sg(\truncnorm)$ is the first zero of $h_0(., a)$.
\item If $a=0$ (resp. $b=0$), then $\sg(\truncnorm)$ is the first zero of $h_1(., b)$ (resp. $h_1(., a)$).
\end{itemize}
\end{prop}

\subsection{Numerical method} \label{sec:numerical_method}

In this section we present a numerical method to estimate the spectral gap.
The technique is well-known in the literature of elliptic problems (see e.g. \cite{Raviart_Thomas} or \cite{Allaire_Book}), and we adapt it to our case.
We consider the framework and notations of Section~\ref{sec:spectral}.
The spectral problem (P3),
$$ f'' - V'f' = - \lambda f, \qquad f'(a) = f'(b) =0 $$
can be solved numerically by using finite element methods.\\

The first step is to consider the variational problem (P2), 
rewritten under the form (see the Proof of Theorem~\ref{prop:SpectralTheorem}):
\begin{equation} \label{eq:weak_formulation2}
\alpha(f, g) = (\lambda+1) \langle f, g \rangle \qquad \forall g \in \Sob{1}{\mu}(\Omega) 
\end{equation}
where $\alpha(f,g) = \langle f', g' \rangle + \langle f, g \rangle$ 
is coercive on $\Sob{1}{\mu}(\Omega)$, and to show its equivalence to (P3). This was done in Theorem~\ref{prop:SpectralTheorem}.\\

\begin{figure} 
\begin{center}
\includegraphics[width=0.75\textwidth]{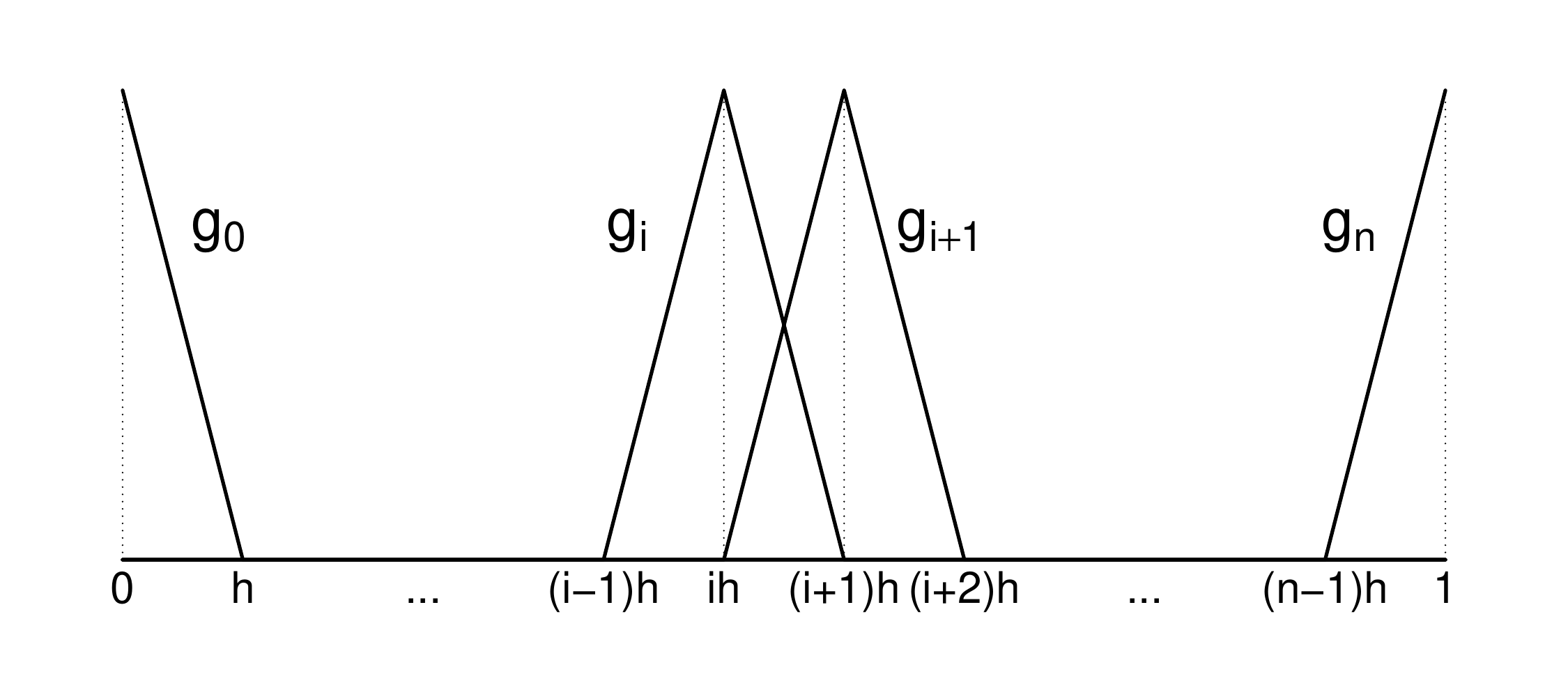}
\caption{Basis of finite elements $\mathbbm{P}_1$ on $[0,1]$. 
The $g_i$'s are hat functions for $i = 1, \dots, n-1$, truncated at the boundaries ($i=0$ and $i=n$).}
\label{fig:ElementsFinisP1}
\end{center}
\end{figure}

The second step is to observe that (\ref{eq:weak_formulation2}) can be solved algebraically in a finite-dimensional space. Denote $G_h$ a finite-dimensional space of $\Sob{1}{\mu}(\Omega)$, 
where $h>0$ is a discretization parameter, and let $(g_i)_{0 \leq i \leq n}$ be a basis of $G_h$.
Typically, $G_h = \mathbbm{P}_1$, the space of Lagrange finite elements, 
composed of piecewise linear functions on $[a, b]$.
A basis is formed by the `hat' functions $g_i$ represented in Figure~\ref{fig:ElementsFinisP1}.
A solution in $G_h$ of (\ref{eq:weak_formulation2}) is given by
$$ f_h = \sum_{i=0}^n f_{h, i} g_i$$
and the weak formulation in $G_h$ can be written in the matricial form:
\begin{equation} \label{eq:KM_matricial_problem}
K_h \mathbf{f_h} = (\lambda+1) M_h \mathbf{f_h}
\end{equation}
where $\mathbf{f_h}$ is the vector of $(f_{h,i})_{0 \leq i \leq n}$, and $K_h$ and $M_h$ are called respectively `rigidity matrix' and `mass matrix', defined by:
$$K_h = (\alpha(g_i, g_j))_{0 \leq i,j \leq n}  \qquad \qquad M_h = (\langle g_i, g_j \rangle)_{0 \leq i,j \leq n}$$
Notice that $M_h$ and $K_h$ are symmetric and positive definite.
Then the problem can be reexpressed in a standard form by using the Choleski decomposition of $M_h = L_h L_h^T$,
where $L_h$ is a lower triangular matrix. 
Indeed, denoting $\widetilde{K_h} = L_h^{-1} K_h (L_h^T)^{-1}$ and $\widetilde{\mathbf{f_h}} = L^T \mathbf{f_h}$, 
(\ref{eq:KM_matricial_problem}) is written:
$$ \widetilde{K_h} \widetilde{\mathbf{f_h}} = (\lambda+1) \widetilde{\mathbf{f_h}} $$
Thus $\lambda$ and $\widetilde{\mathbf{f_h}}$ are obtained by performing the eigen decomposition of the symmetric matrix $\widetilde{K_h}$. Finally $\mathbf{f_h}$ is deduced from the relation $\mathbf{f_h} = (L_h^T)^{-1} \widetilde{\mathbf{f_h}}$.\\

The last step is to observe that the solutions of the finite-dimensional weak formulation (\ref{eq:KM_matricial_problem}) converge to the solutions of (P2), and equivalently (P3),
when $h \rightarrow 0$, as stated in \cite{Raviart_Thomas} (Theorem~6.5.1.), \cite{Allaire_Book} (\S~6.2.2.), or \cite{Babuska_Osborn_inHandbook} (\S~8, Eq. 8.43 and \S~10.1.2. Eq. 10.22), with a speed of convergence linked to the regularity of the solutions. These results are expressed for the Lebesgue measure on $[a,b]$, but are still valid under the assumptions of Theorem~\ref{prop:SpectralTheorem} since for all integer $\ell$, $\Sob{\ell}{\mu}(\Omega)$ is then equal to $\Sob{\ell}{\mathrm{Leb}}(\Omega)$ with equivalent norms (see the proof of Theorem~\ref{prop:SpectralTheorem}).
In our situation, the space spanned by the first two eigenfunctions lies in $\Sob{\ell+1}{\mu}(\Omega)$, with $\ell \geq 1$ (see Theorem~\ref{prop:SpectralTheorem}).
Hence the smallest strictly positive eigenvalue of (\ref{eq:KM_matricial_problem}), i.e. the second one, converges to the spectral gap at the speed $O(h^{2\ell})$,
and -- since it is a simple eigenvalue -- the corresponding eigenvector converges to a function saturating the Poincar\'e inequality at the speed $O(h^\ell)$.

\begin{rem}[Numerical improvements]
We briefly mention two well-known improvements for finite elements algorithms, which apply to our case (see e.g. \cite{Allaire_Book}). Firstly, it is possible to replace the computations of integrals in $M_h$ by quadrature formulas (`mass lumping' speed-up technique). This allows obtaining a \emph{diagonal} matrix and $L_h$ is simply the square root of its diagonal terms. The convergence result is still valid. Secondly, if the solution is regular enough (typically $f \in \Sob{3}{\mu}(\Omega)$), more regular basis functions can be used (such as $\mathbbm{P}_2$ finite elements). The speed of convergence is higher, at the price of a higher computational cost.
\end{rem}

\subsection{Illustrations}

We illustrate the results of the previous paragraphs on the truncated double exponential and truncated normal distributions.
For the former, the semi-analytical result has been used. The numerical method was also tested, and gives a so good approximation that the difference between the two methods could not be seen in Figure~\ref{fig:TruncExp}.
For the latter, however, the numerical method based on finite elements has been prefered, 
rather than searching for the first zero of Kummer's functions. 
Indeed, in practice, it is not easy to find numerically the \emph{first} zero of a given function. 
In particular, it was not possible to justify theoretically the empirical observation that the first zero is the only zero lying in the interval
$]0, \sigma_I^{-2}[$, where $\sigma_I^2$ is the variance of the truncated distribution on $I$.
On the other hand, the numerical method of Section~\ref{sec:numerical_method} is theoretically sound.
Implementations have been done with the R software \cite{R} and the R package `orthopolynom' \cite{HermitePackage}.

\subsection*{Truncated exponential distribution}
Denote by $\mu$ the double exponential distribution.
Figure~\ref{fig:TruncExp} shows the Poincar\'e constant of $\mu$ truncated on $I = [a, b]$.
In the general case, we plotted the contour lines of $\poinc(\mu)$ 
as a function of $F_\mu(a)=P(X \leq a)$ and $1 - F_\mu(b) = P(X \geq b)$, where $X\sim \mu$. 
In this representation, symmetric intervals correspond to the diagonal $y=x$, 
intervals $(-\infty, 0]$ to $y=\frac12$, and intervals $[0, +\infty)$ to $x = \frac12$.
We also added the upper bound of Prop.~\ref{prop:trunc_exp}, 
which is an exact bound in the regions $b \leq 0$ and $a \geq 0$, 
i.e. $1 - F_\mu(b) \geq \frac12$ and $F_\mu(a) \geq \frac12$.

For symmetric intervals, $\poinc(\mu)$ is represented as a function of the mass $\mu(I)$, 
as well as two upper bounds and the lower bound corresponding to the variance of $\mu$ on $I$, 
$\sigma_I^2 = 2 \left(1 - e^{-b} \left( 1+b+\frac{b^2}{2} \right) \right) \left/ F_\mu(b) \right.$.
We can see that the upper bound obtained by symmetric transport $4\mu(I)^{2}$
(Lemma~\ref{lem:TransportSymmetricInterval}) is quite large here, except for strong truncations.
The upper bound of Prop.~\ref{prop:trunc_exp} is much sharper.
The lower bound is also very sharp, except for small truncations, 
since it tends to $2$ when $\mu(I)$ tends to $1$ whereas the Poincar\'e constant tends to $4$.

\begin{figure}[h!]
\centering
\includegraphics[width=0.44\linewidth]{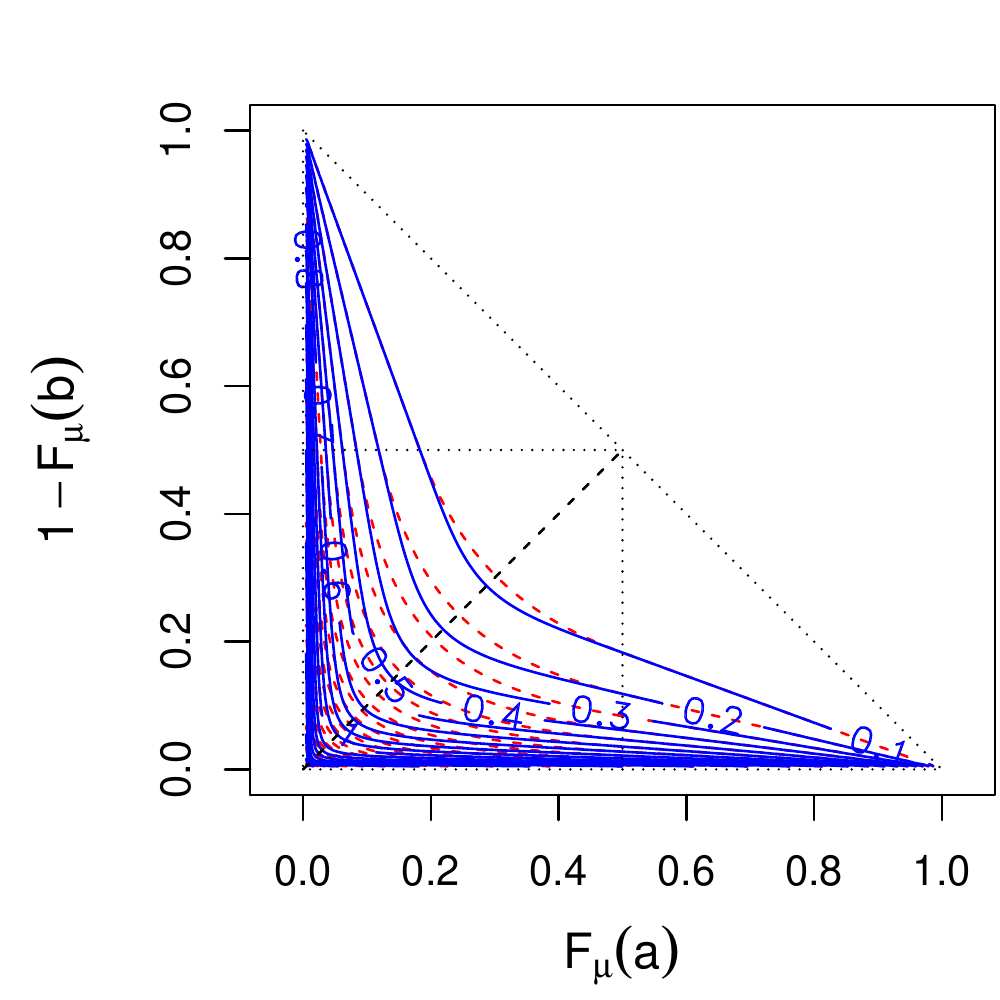}
\includegraphics[width=0.55\linewidth]{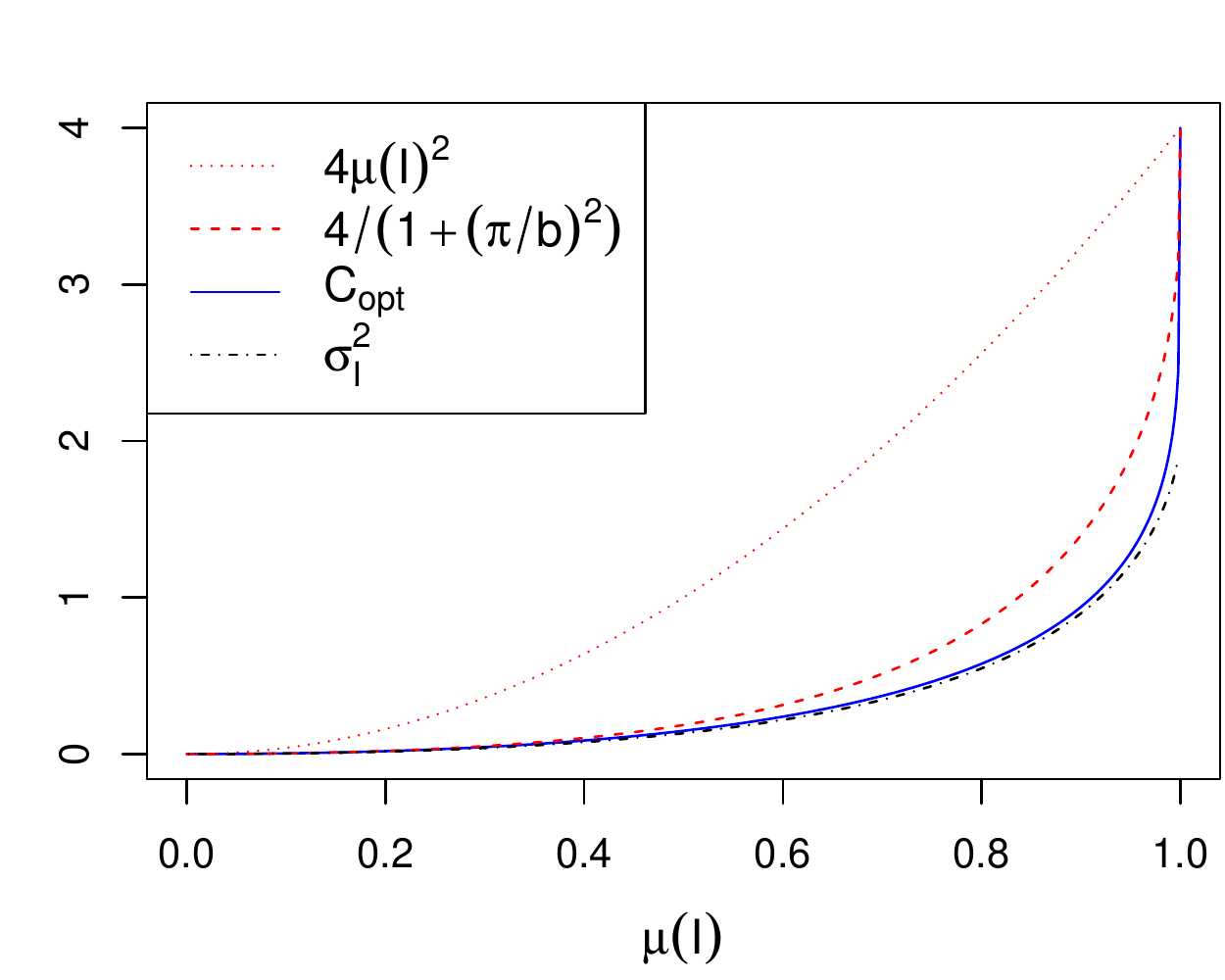}
\caption{Poincar\'e constant of the double exponential distribution $\mu$ truncated on $[a,b]$.
Left: Contour plot as a function of $F_\mu(a)$ and $1-F_\mu(b)$, with the upper bound
$\left( \frac14 + \left( \frac{\pi}{b-a} \right)^2 \right)^{-1}$ (red dashed lines). 
Right: Representation in the symmetric case for $I = [-b, b]$, with two upper bounds 
and the lower bound given by the variance.}
\label{fig:TruncExp}
\end{figure}

\subsection*{Truncated normal distribution}
Similar plots have been produced for the normal distribution $N(0,1)$, truncated on $I = [a,b]$,
gathered in Figure~\ref{fig:PoincareNormal}.
In the symmetric case, the upper bound obtained by transport $\mu(I)^{2}$
(Lemma~\ref{lem:TransportSymmetricInterval}) looks globally accurate, 
especially for strong truncations. 
The lower bound given by the variance of the truncated distribution $\sigma_I^2 = 1 - 2 \frac{b \phi(b)}{\mu(I)}$ 
is even sharper.
In order to visualize the link to Hermite polynomials, we have added the points corresponding to
intervals $[-b,b]$ where $-b,b$ are two successive zeros of Hermite polynomials of degree $2n$, up to degree $100$.
Recall that in that case $\poinc(\mu) = 1/(2n+1)$. 

In the general case, 
we have added the intervals $[r_{n, i-1},r_{n,i}]$ corresponding to consecutive zeros of Hermite polynomials of degree $n$ (up to $n = 100$), associated to Poincar\'e constants equal to $1/(n+1)$.
We can observe that the whole set of Hermite consecutive zeros poorly fill the space of possible intervals,
except for very small symmetric intervals (located around the diagonal $x+y=1$).
The Poincar\'e constants obtained with their extensions, namely Kummer's functions, 
thus provide a useful improvement in many cases.

\begin{figure}[h!]
\centering
\includegraphics[width=0.44\linewidth]{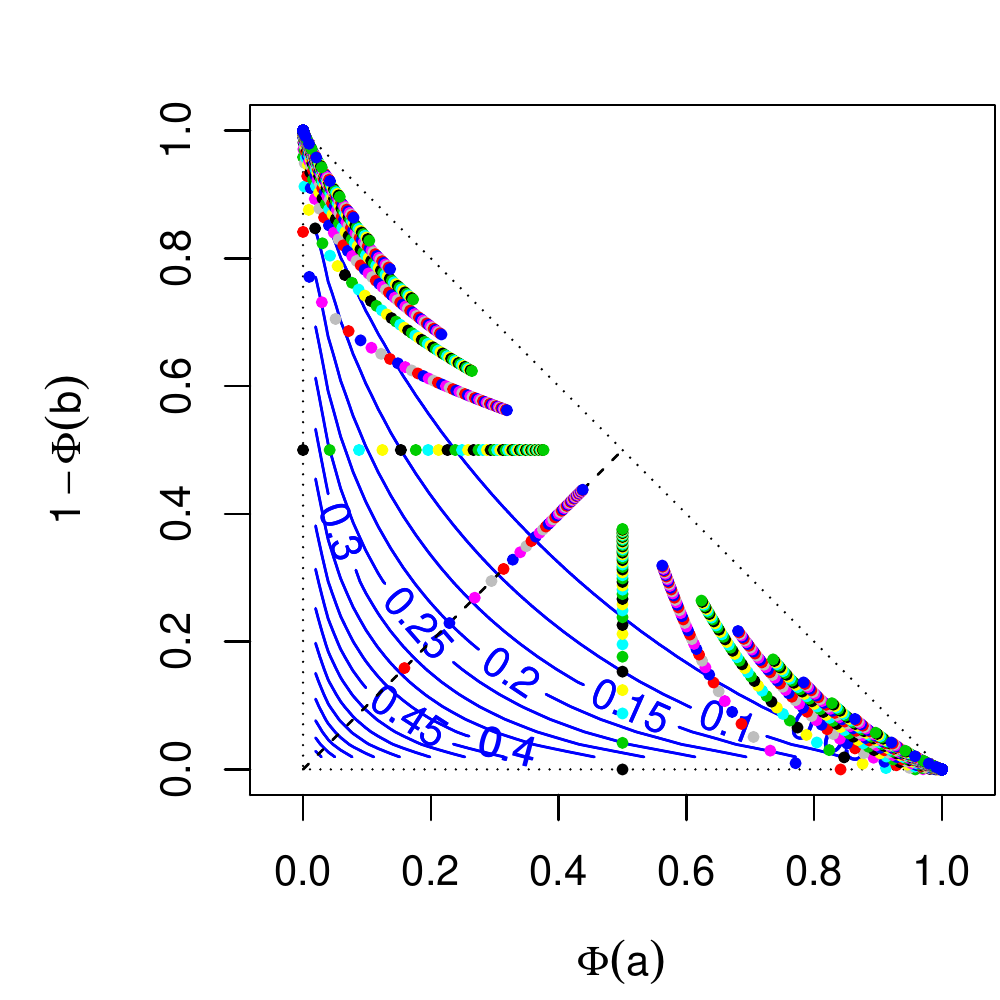}
\includegraphics[width=0.55\linewidth]{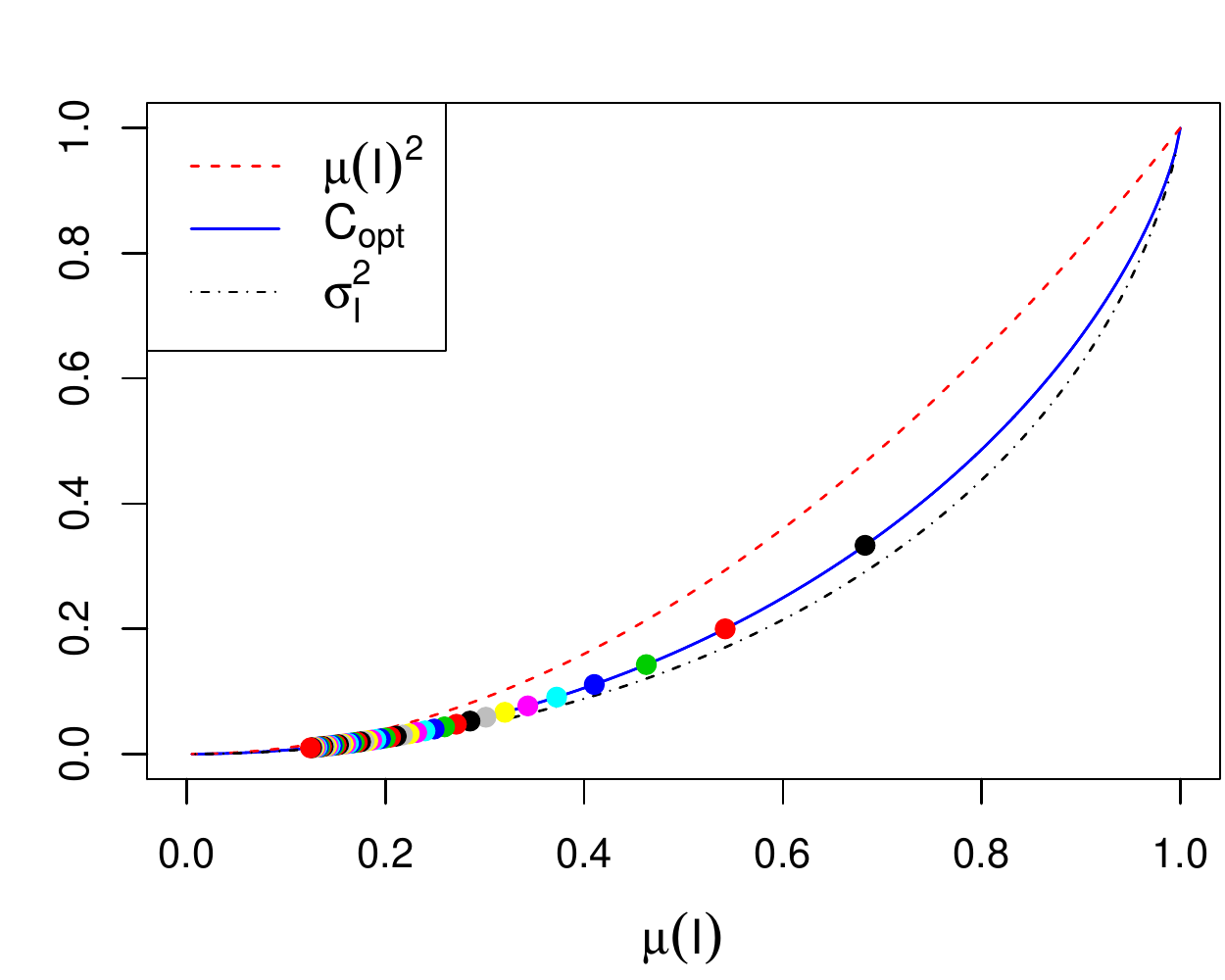}
\caption{Poincar\'e constant of the normal distribution $N(0,1)$ truncated on $[a,b]$.
Left: Contour plot as a function of $\Phi(a)$ and $1-\Phi(b)$. 
Right: Representation in the symmetric case for $I = [-b, b]$, 
with the upper bound $\mu(I)^2$ and the lower bound given by the variance. 
Colored points correspond to Hermite polynomials: For a given degree, a same color is used.}
\label{fig:PoincareNormal}
\end{figure}

\section{Applications} \label{sec:applications}

In this section, we perform some global sensitivity analysis by computing the DGSM-based upper bounds of the total Sobol indices (see Equation \ref{eq:PoincareTypeIneq}), on two hydraulic applications based on different models.
For these applications, the common problem under investigation is the assessment of the water level in the terminal section of a watercourse in case of flood. 
The water level is evaluated as a function of the discharge and other physical parameters which will be detailed hereinafter. 
Uncertainties on these physical parameters are modelled by probability distributions, chosen from expert knowledge and in accordance with the empirical distributions obtained during measurement campaigns.
Flood phenomenon is governed by the Saint Venant shallow water equations, connecting the water level localized in space and time, the discharge, the water section, the lateral inflows, the slope and the head losses due to the friction between the water body and the riverbed.
The first application relies on a coarse simplification of these equations while the second application will be based on a 1D solving model.

The computations of Poincar\'e constants have been done using the R package `sensitivity' \cite{SensitivityPackage} of the R software \cite{R}.
In this package, the function `PoincareConstant()' allows to compute the upper bounds given by the double exponential transport and logistic transport, while the function `PoincareOptimal()' provides the Poincar\'e constant by solving numerically the spectral problem.

\subsection{First study on a simplified flood model}

This application aims at simulating the height of a river and compares it to the height of a dyke that protects industrial facilities. 
This academic model has been introduced for a pedagogical purpose in several methodological papers (see for example \cite{Iooss_Lemaitre_review}), and used for illustrating DGSM-based sensitivity analysis in \cite{Lamboni_et_al_2013} and \cite{RoustantCrossDerivative}.

In the case of steady flow, with no inflows, and large rectangular section, assuming that the classical Manning-Strickler formulation is used for the head losses, we obtain the closed-form solution:
\begin{equation}\label{StricklerFormulaH}
S = H + Z_v - H_d - C_b \quad \mbox{with} \quad H = \left(\frac{Q}{BK_s \sqrt{\frac{Z_m-Z_v}{L} }} \right)^{0.6} 
\end{equation}
where the output variable $S$ is the maximal annual overflow (in meters), $H$ is the maximal annual height of the river (in meters), $Q$ is the river flowrate, $K_s$ the Strickler's friction coefficient, $Z_m$ and $Z_v$ the upstream and downstream riverbed levels (w.r.t. a fixed reference), $L$ and $B$ are the length and width of the water section, $H_d$ is the dyke height and $C_b$ is the bank level.
Table \ref{tab:factors} gives the probability distributions of the model input variables which are supposed to be independent. Notice that, as in the rest of this paper, we have used here the notation ${\mathcal N}(\mu, \sigma^2)$ where $\sigma$ is the standard deviation, contrarily to \cite{Lamboni_et_al_2013} which uses ${\mathcal N}(\mu, \sigma)$. Hence the standard deviation of $K_s$ is equal to $8$.

\begin{table}[!ht]
  \begin{center}
   \begin{tabular}{lccr}
Input & Description & Unit & Probability distribution \\
   \hline
 $X_{1} = Q$ & Maximal annual flowrate & m$^3$/s & Gumbel ${\mathcal G}(1013, 558)$\\ 
  & & &  truncated on $[500 , 3000 ]$  \\  
 $X_{2}=K_s$ & Strickler coefficient & - & Normal ${\mathcal N}(30, 8^2)$ \\
  & & &  truncated on $[15 , +\infty [$  \\
 $X_{3} = Z_v$ & River downstream level & m & Triangular  ${\mathcal T}(49, 51)$ \\
 $X_{4} = Z_m$ & River upstream level  & m  & Triangular  ${\mathcal T}(54, 56)$  \\
 $X_{5} = H_d$ & Dyke height & m &  Uniform ${\mathcal U}[7, 9]$ \\
 $X_{6} = C_b$ & Bank level  & m & Triangular  ${\mathcal T}(55, 56)$ \\
 $X_{7} = L$ & River stretch  & m &  Triangular  ${\mathcal T}(4990, 5010)$  \\
 $X_{8} = B$ & River width  & m &  Triangular  ${\mathcal T}(295, 305)$ \\
\hline
    \end{tabular}
  \end{center}
    \caption{Input variables of the flood model and their probability distributions.}       \label{tab:factors}
\end{table}

Table \ref{tab:app1} gives the Poincar\'e constants and two upper bounds for the different probability distributions used in this test case (the uniform distribution is not represented as the result is well known). 
For the sake of interpretation, we have considered scaled distributions. 
Indeed, a simple change of variable (which can also be viewed as a linear transport) shows that: 
\begin{eqnarray*}
C_P \left( \mathcal{N} (\mu, \sigma^2) \vert [a,b] \right) &=& \sigma^2 C_P \left( \mathcal{N}(0, 1) \left\vert  \left[ \frac{a-\mu}{\sigma}, \frac{b-\mu}{\sigma}\right]\right. \right), \\
C_P \left( \mathcal{G} (\mu, \beta) \vert [a,b] \right) &=& \beta^2 C_P \left( \mathcal{G}(0, 1) \left\vert  \left[ \frac{a-\mu}{\beta}, \frac{b-\mu}{\beta}\right]\right. \right), \\
C_P(\mathcal{T}(a,b)) &=& \left( \frac{b-a}{2} \right) ^2 C_P(\mathcal{T}(-1, 1)),
\end{eqnarray*}
where the notation $|I$ means that the distribution is truncated on the interval $I$.
 We can see the strong decrease factor between the Poincar\'e constant and the upper bounds based on double exponential transport (gain factor around $6$) or logistic transport (gain factor around $2$).

\begin{table}[!ht]
  \begin{center}
	\begin{tabular}{ccccc}
		Prob. dist. &  Upper bound & Upper bound & $\poinc(\mu)$ & Lower bound\\ 
		$\mu$ & (db. exp. transp.) & (logis. transp.) &  & $\var(\mu)$ \\ \hline
		$\mathcal{T}(-1,1)$ & $1$ & $0.296$ & $0.173$ & $0.167$\\ \hline
		$\mathcal{N} \left(0, 1 \left\vert \right. [-1.87, +\infty)  \right)$ & $5.912$ & $1.484$ & 0.892 & $0.862$\\ \hline
	$\mathcal{G} \left(0,1 \left\vert \right. [-0.92, +3.56] \right)$ & $6.956$ & $2.418$ & $1.257$ & $1.012$ \\ \hline 			
	\end{tabular}
  \end{center}
	\caption{Poincar\'e constants and bounds for the scaled laws used in the simplified flood model.}\label{tab:app1}
\end{table}

Figure \ref{fig:app1} gives the final global sensitivity analysis results for this test case.
It is based on the numerical values obtained in \cite{Lamboni_et_al_2013}.
In particular, the DGSM $\nu_i$ ($i=1,\ldots,8)$ have been computed via a sample of the output derivatives coming from a low discrepancy sequence of the inputs of size $10\,000$.
The previous results of \cite{Lamboni_et_al_2013} (given here in Fig. \ref{fig:app1}, left)  have shown that the DGSM-based upper bounds can be used for a screening purpose ({\it i.e.} identifying non influential inputs which have a total Sobol index close to zero), but are useless for quantifying the effects of the influential inputs (because of their non-informative high values).
In contrary, our new results (Fig. \ref{fig:app1}, right) show that the DGSM-based upper bounds give a reliable information on the real influence of the inputs (in terms of contribution to the model output variance) thanks to their closeness to total Sobol indices.
The hierarchy of influence given by the DGSM-based upper bounds is the same than 
for total Sobol indices: the flowrate $Q$ is the most influential input, followed by $H_d$, $Z_v$ and $K_s$. 
The remaining four inputs are not active in the model.
We can also note that screening is now fully efficient as $C_b$ can be judged as non influential (by choosing for example a threshold of $5\%$ of the total variance for the sensitivity index). 

\begin{figure}[!ht]
\centering
\includegraphics[width=0.49\linewidth]{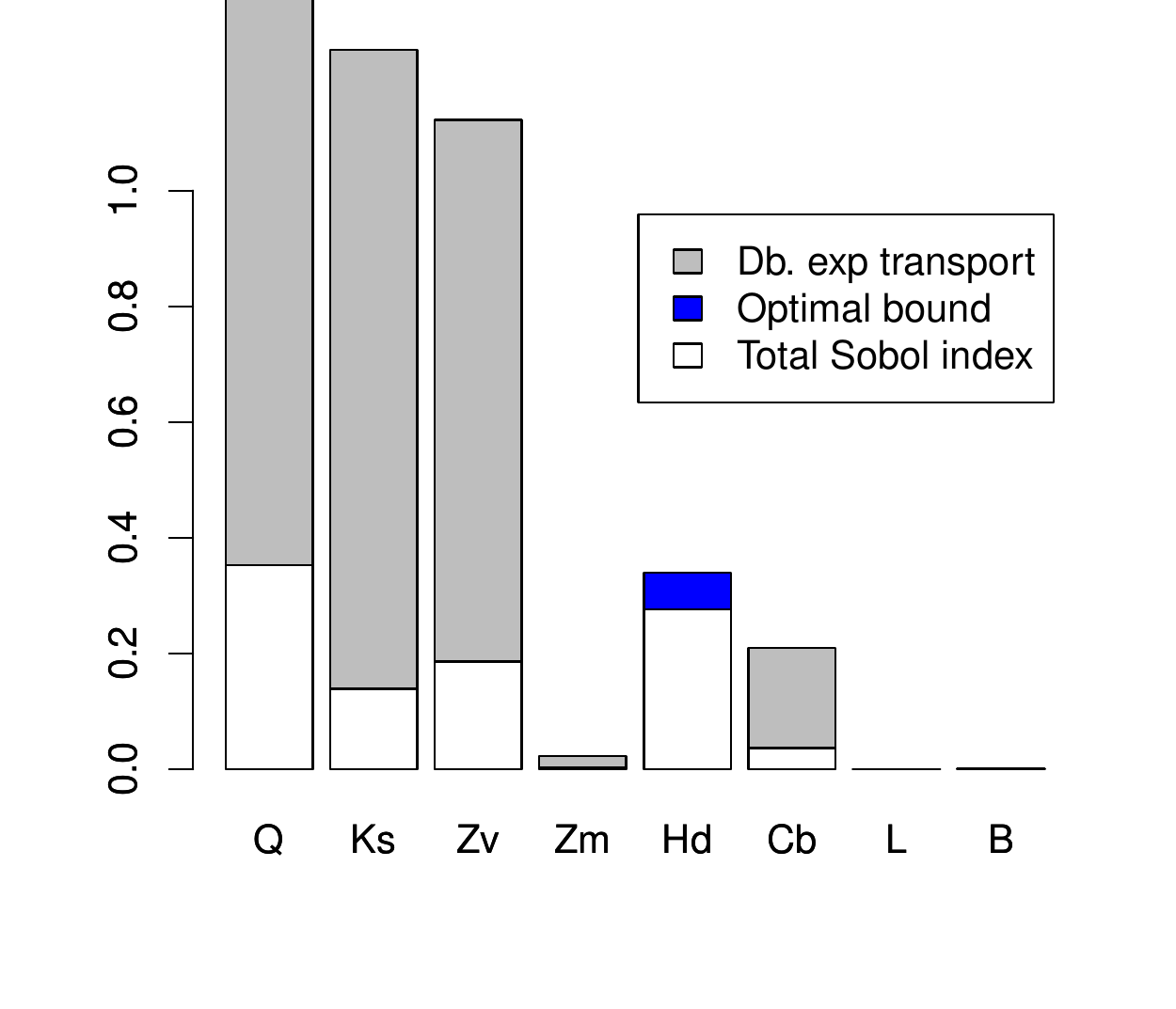} \includegraphics[width=0.49\linewidth]{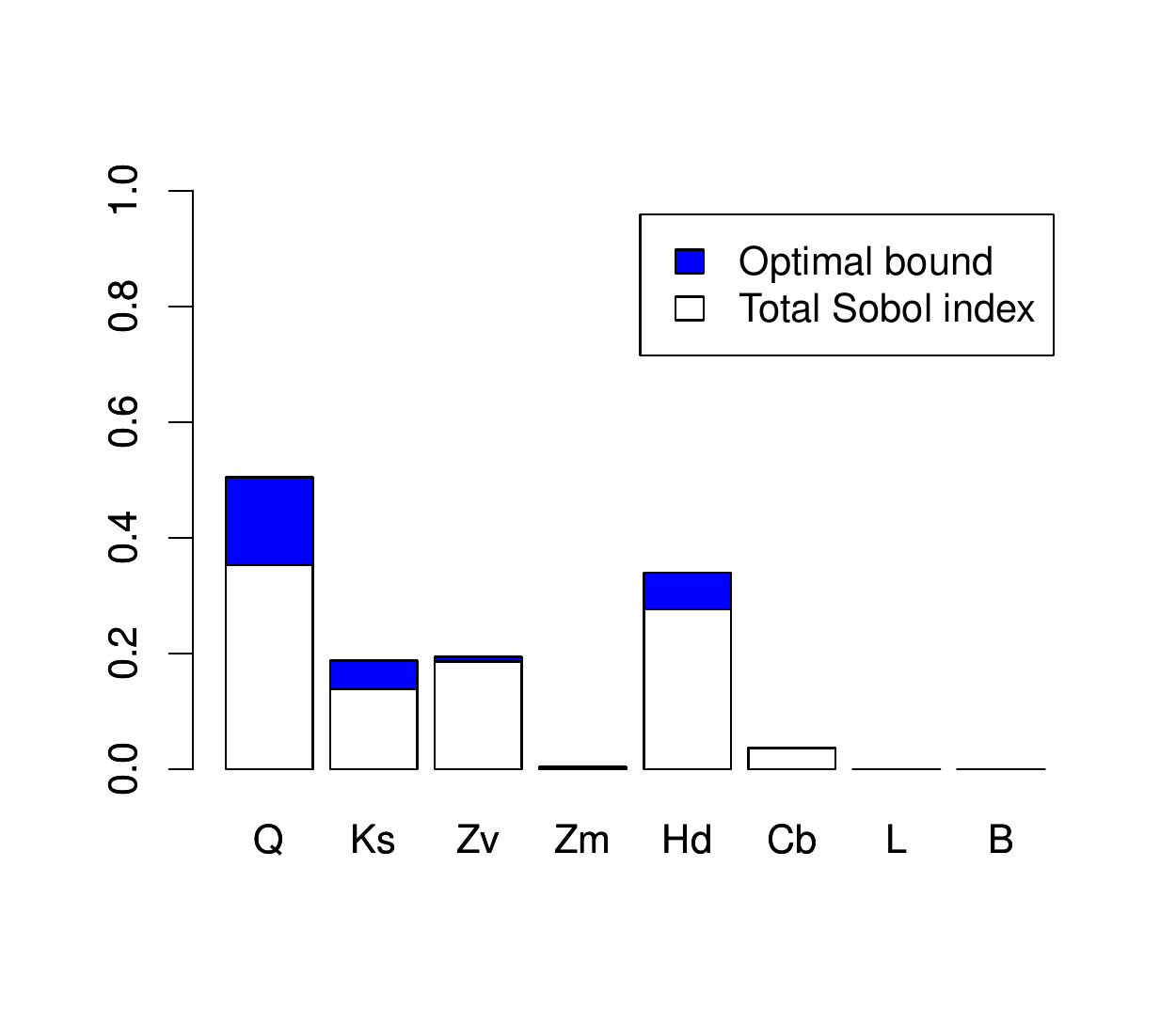}
	\caption{Comparison between total Sobol indices and DGSM-based upper bounds for the simplified flood model. Left: Results with double exponential transport; Right: New optimal bounds, obtained with Poincar\'e constants.}\label{fig:app1}
\end{figure}

\begin{rem}
For the DGSM-based upper bound of $K_s$, our value (Table \ref{tab:app1}, left) is different from the one of \cite{Lamboni_et_al_2013} (Table 5).
Indeed, an error is present in \cite{Lamboni_et_al_2013} where the multiplicative factor $2\pi$ has been omitted (see also the erratum in \cite{RoustantCrossDerivative}, Remark 3).
Hence, in \cite{Lamboni_et_al_2013}, $0.198$ has to be replaced by $1.244$.
\end{rem}

\subsection{Application on a 1D hydraulic model}

The Mascaret computer code (\cite{Goutal_et_al_2012}) is a computer code based on a 1D solver of the Saint Venant equations, allowing engineers to calculate water height for river flood events.
The studied case, taken from \cite{Petit_et_al_2016} and illustrated in Figure \ref{fig:lit}, is the French Vienne river in permanent regime whose input data are an upstream flowrate, a downstream water level, physical parameters (Strickler coefficients) and geometrical data.
The geometrical data consist of $12$ transverse river profiles.
In summary, the model includes the following random input variables, assumed independent:
\begin{itemize}
\item{$12$ Strickler's friction coefficients of the main channel $K_{s,c}$ (noted $K^1_{s,c}$, \ldots, $K^{12}_{s,c}$) whose distributions are uniform on $[20,40]$,}
\item{$12$ Strickler's friction coefficients of the flood plain $K_{s,p}$ (noted $K^1_{s,p}$, \ldots, $K^{12}_{s,p}$) whose distributions are uniform on $[10,30]$,}
\item{$12$ slope perturbations $dZ$ (noted $dZ^1$, \ldots, $dZ^{12}$) whose distributions are standard Gaussian, truncated on $[-3,3]$,}
\item{and one discharge value $Q$ whose distribution is Gaussian with zero mean, standard deviation $50$ and truncated on $[-150,150]$.}
\end{itemize}

\begin{figure}[!ht]
\centering
\includegraphics[scale=0.25]{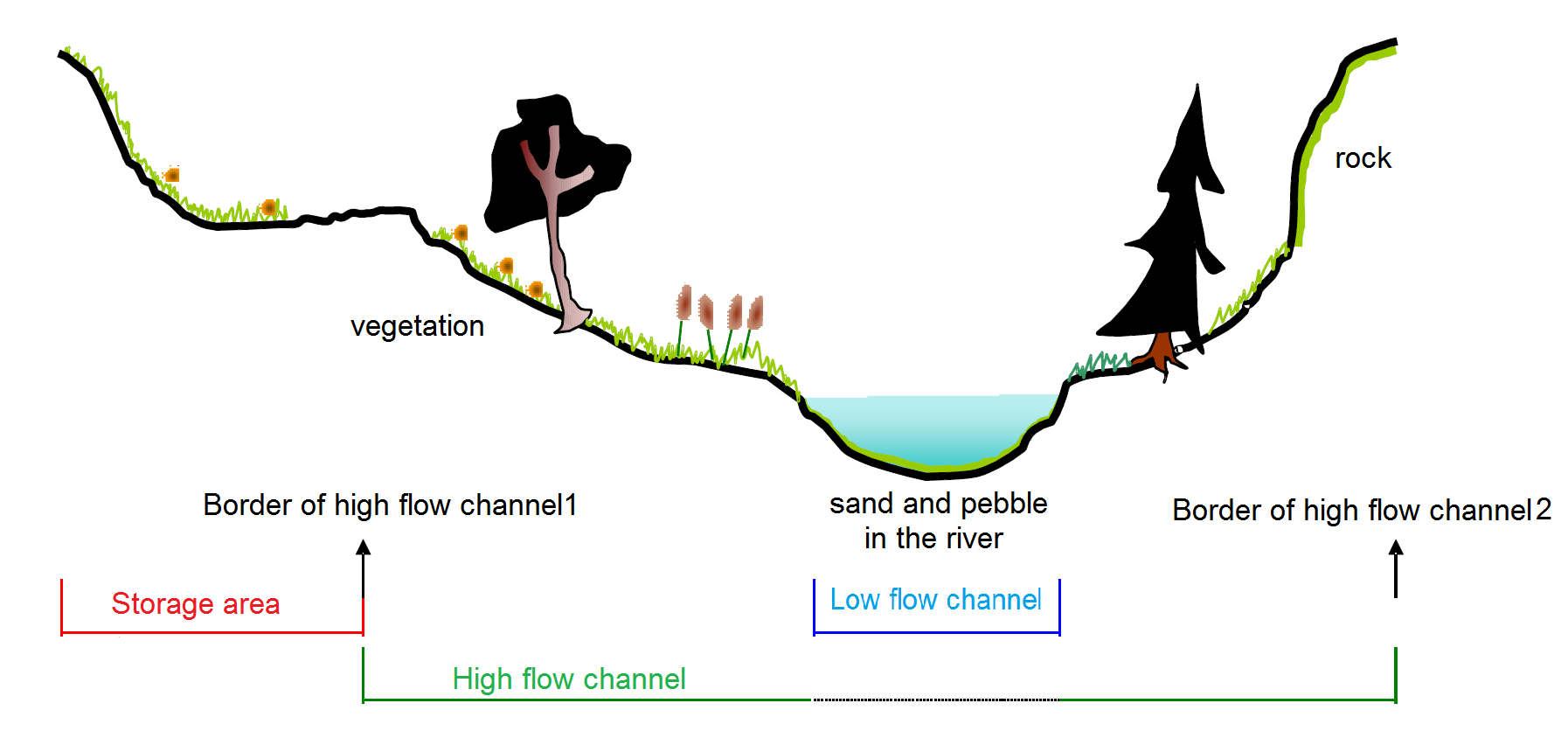}
	\caption{Representation of a river cross section, showing the main channel (low flow channel) and the flood plain (high flow channel). Source: \cite{Petit_et_al_2016}.}\label{fig:lit}
\end{figure}

Our goal in this test case is to update with the new DGSM bounds the work of \cite{Petit_et_al_2016} who compare DGSM bounds and Sobol indices. 
In this previous work, the derivatives of the model output with respect to the $37$ inputs can be efficiently (with a cost independent of the number of inputs) computed by using the adjoint model of Mascaret (obtained by automatic differentiation, \cite{Griewank_Walther_2008}).
From a Monte Carlo sample of the inputs of size $n=20\,000$, $n$ output values (water height) and $n$ output derivatives are obtained.
This very large $n$ has been used for a demonstrative purpose, while in other industrial studies $n$ ranges from $100$ to $1\,000$ (\cite{Iooss_et_al_2012,Touzany_Busby_2014}). 
From the sample of derivatives, the $37$ DGSM $\nu_i$ are then computed.
\cite{Petit_et_al_2016} has shown that $32$ inputs have DGSM-based upper bounds close to zero.
Then, amongst the $37$ inputs, only $5$ are potentially active.

Table \ref{tab:app2} gives our new results for these $5$ remaining inputs.
To compute the DGSM bound, we have also computed $\var(f(\mb{x}))=0.369$ whose standard deviation ($sd=3.4e{-3}$) is obtained by bootstrap, {\it i.e.} by resampling with replacement the Monte Carlo sample of output values.
As inputs $K^{11}_{s,c}$ and $K^{12}_{s,c}$ follow a uniform distribution, we report only the optimal DGSM bounds, already used in \cite{Petit_et_al_2016}.
For the $3$ other inputs ($dZ^{11}$, $dZ^{12}$ and $Q$), as in the previous application, a strong decrease factor is obtained for the bounds by using the logistic constant (gain factor around $4$) and the optimal constant (gain factor around $6$) instead of the double exponential constant.
The new bounds now allow to identify the variable $dZ^{12}$ as inactive.
Moreover, they provide the same ranking than total Sobol indices for the influential inputs.
Notice that the optimal bounds are nearly equal to the total Sobol indices.
The numerical errors, assessed by the standard deviations of the estimates, explain that the bound values are sometimes slightly smaller than the Sobol values.

\begin{table}[!ht]
  \begin{center}
	\begin{tabular}{cccccc}
		Inputs &  $K^{11}_{s,c}$ & $K^{12}_{s,c}$  & $dZ^{11}$ & $dZ^{12}$ & $Q$\\ \hline
		$S^T$ & $0.456$ & $0.0159$ & $0.293$ & $0.015$ & $0.239$ \\ 
		      & $(2e{-3})$ & $(1e{-4})$ & $(1e{-3})$ & $(1e{-4})$ & $(1e{-3})$ \\ \hline
		$\nu$ & $5.695e{-3}$ & $2.728e{-4}$ & $1.089e{-1}$ & $6.592e{-3}$ & $3.553e{-5}$ \\ 
		      & $(3e{-5})$ & $(4e{-6})$ & $(3e{-4})$ & $(9e{-5})$ & $(6e{-8})$ \\ \hline
		\multicolumn{6}{c}{By double exponential transport}\\
		Upper bound for $\poinc(\mu)$ & - & - & $6.249$ & $6.249$ & $15623.26$\\ 
		Upper bound for $S^T$ & - & - & $1.844$ & $0.116$ & $1.504$\\ 
		      & - & - & $(2e{-3})$ & $(2e{-3})$ & $(1.5e{-2})$\\ \hline
		\multicolumn{6}{c}{By logistic transport}\\
		Upper bound for $\poinc(\mu)$ & - & - & $1.562$ & $1.562$ & $3905.815$\\ 
		Upper bound for $S^T$ & - & - & $0.461$ & $0.028$ & $0.376$\\ 
		      & - & - & $(4e{-3})$ & $(5e{-4})$ & $(4e{-3})$\\ \hline
		\multicolumn{6}{c}{Optimal, with the Poincar\'e constant}\\
		$\poinc(\mu)$ & $40.528$ & $40.528$ & $0.976$ & $0.976$ & $2441.071$\\ 
		Optimal bound for $S^T$ & $0.625$ & $0.029$ & $0.288$ & $0.017$ & $0.235$\\ 
		      & $(2e{-4})$ & $(1e{-5})$ & $(3e{-3})$ & $(3e{-4})$ & $(2e{-3})$\\ \hline
		\end{tabular}
  \end{center}
	\caption{Sensitivity indices for the Mascaret test case. $S^T$ gives the total Sobol index, $\nu$ the DGSM. Standard deviations of all these estimates are obtained by bootstrap and given in parentheses. Partial results are not shown for $K^{11}_{s,c}, K^{12}_{s,c}$, which are uniformly distributed.}  
	\label{tab:app2}
\end{table}

In terms of interpretation, the sensitivity analysis results show that the hydraulic engineers have to concentrate the research efforts on the knowledge of the physical parameters $K_s$ and $dZ$ at the $11$th river profile, in order to be able to reduce the prediction uncertainty of the water level when a flood occurs.
From a methodological point of view, the numerical model users are now able to perform a global sensitivity analysis, in the sense of variance decomposition, at a lower cost than before by using the DGSM-based technique.

\section*{Acknowledgements}
We are grateful to Fabrice Gamboa, who initiated this research.
We thank EDF R\&D/LNHE for providing the Mascaret test case and S\'ebastien Petit who has performed the computations on this model.
We also thank Laurence Grammont and the members of the team `G\'enie Math\'ematique \& Industriel' for useful discussions. 
We acknowledge the participants of MascotNum 2016 \& Mexico MascotNum conferences for their feedbacks. 
Part of this research was conducted within the frame of the Chair in Applied Mathematics OQUAIDO, gathering partners in technological research (BRGM, CEA, IFPEN, IRSN, Safran, Storengy) and academia (CNRS, Ecole Centrale de Lyon, Mines Saint-Etienne, University of Grenoble, University of Nice, University of Toulouse) around advanced methods for Computer Experiments.

\section*{Appendix} 
\begin{proof}[End of proof of Theorem~\ref{prop:SpectralTheorem}]
Let us show (\ref{eq:IntegralFormOfPrime}), namely that $\mu$-almost surely,
\begin{equation} 
f'(x) = \frac{\lambda}{\pdf(x)} \int_x^b f(t)\pdf(t)dt.
\end{equation}
Following \cite{Bobkov2009}, Lemma 4.3., by using $g(x) = g(a) + \int_a^x g'(t)dt$ and Fubini's theorem, it holds:
$$ \int_a^b f(x)g(x)\pdf(x)dx = g(a) \int_a^b f(x)\pdf(x)dx + \int_a^b \left( \int_x^b f(t)\pdf(t) dt \right) g'(x)dx$$ 
Recall that for $\lambda>0$, $\int_a^b f \pdf=0$. 
Thus (P2) can be rewritten:
$$ \int_a^b f'(x)g'(x)\pdf(x)dx = \int_a^b \left( \lambda \int_x^b f(t)\pdf(t) dt \right) g'(x)dx $$ 
Since it is true for all $g'$ in $L^2(\mu)$ this gives (\ref{eq:IntegralFormOfPrime}).\\

It remains to show that $f$, a minimizer of the Rayleigh ratio, is strictly monotonic. 
The main arguments can be found e.g. in \cite{Bobkov2009}. Define
$$ g(x) = \int_a^x \vert f' \vert d\mu$$
Firstly, we have $g' = \vert f' \vert$ and thus $\int (g')^2 d \mu = \int (f')^2 d\mu$. Secondly, using the formula $\var_{\mu} g = \frac 12 \iint [g(y) - g(x)]^2 d\mu(x)d\mu(y)$, we have :
\begin{eqnarray*}
\var_{\mu} g &=& \frac12 \iint  \left( \int_x^y \vert f'(t) \vert d\mu(t) \right)^2 d\mu(x)d\mu(y) \\
&\geq& \frac12  \iint \left( \int_x^y f'(t)  d\mu(t) \right)^2 d\mu(x)d\mu(y) = \var_{\mu} f
\end{eqnarray*}
Since $f$ is a minimizer of the Rayleigh ratio, the above inequality must be an equality, 
leading to $\int_x^y \vert f'(t) \vert d\mu(t) = \vert \int_x^y f'(t)  d\mu(t) \vert$ for all $x,y$ in $[a,b]$.
By continuity of $f'$, there exists $\epsilon \in \{-1, +1\}$ ($\epsilon = \textrm{sgn}[f(b) - f(a)]$) such that $\vert f' \vert = \epsilon f'$ everywhere, hence $f$ is monotonic.\\
To see that $f$ is \emph{strictly} monotonic, consider again (\ref{eq:IntegralFormOfPrime}).
Assume for instance that $f$ is increasing ($f' \geq 0$), and let us prove that $f$ is strictly increasing ($f' > 0$).   
Then, as it is centered, the increasing function $f$ is first $\leq 0$ and then $\geq 0$. 
From (\ref{eq:IntegralFormOfPrime}), it implies that $f'$ is first increasing
from $f'(a)=0$ and then decreasing to $f'(b)=0$.
Furthermore, as $f$ is increasing and centered, we must have $f(a) < 0$ (otherwise $f \geq 0$ on $[a,b]$ and hence identically zero).
By continuity, $f<0$ in a neighborhood of $a$.
Hence from (\ref{eq:IntegralFormOfPrime}) 
$f'$ is strictly increasing in a neighborhood of $a$.
Similarly $f'$ is strictly decreasing in a neighborhood of $b$. 
Finally $f' > 0$ on $(a,b)$.\\

For completeness, we sketch the standard way to derive (P2) from (P1) by using calculus of variation. Without loss of generality we assume that $\int f^2 d\mu = 1$. (P1) is then equivalent to:
$$ \min_f \int_a^b \frac12 (f')^2 d\mu \qquad s.t. \int_a^b f^2d\mu = 1, \quad  \int_a^b f d\mu=0 $$ 
Define $L(t, f, f') = (\frac12 f'^2 - \frac12 \lambda (f^2-1) - \beta f) e^{-V(t)}$. The Lagrangian is:
$$ J(f, \lambda, \beta) = \int_a^b L(t, f(t), f'(t)) dt$$
Considering a small deviation function $g$, we have
\begin{eqnarray*}
J(f+ \epsilon g, \lambda, \beta) &=& J(f, \lambda, \beta) + \\
&& \epsilon \int_a^b \left( g(t)\frac{\partial L}{\partial f}(t, f(t), f'(t)) + g'(t)\frac{\partial L}{\partial f}(t, f(t), f'(t))\right) dt + O(\epsilon^2)
\end{eqnarray*}
The first order condition leads to cancel the integral,
which gives here:
$$\int_a^b \left( -\lambda f g + f'g'\right)d\mu = \beta \int_a^b g d\mu$$
Now, with $g=1$ and using $\int_a^b f d\mu=0$, we get $\beta=0$. This gives (P2).
\end{proof}

\subsection*{Optimal constants}
\begin{proof}[Proof of Proposition~\ref{prop:trunc_exp}] 
\bigskip
a) Let us start by the case $0 \leq a < b$. \\
The spectral problem is given by:
$$ f'' - f' = -\lambda f $$
The characteristic equation $r^2 - r = - \lambda$ can be written as: 
$$\left( r - \frac12 \right)^2 = - \omega^2 \qquad \textrm{with} \quad \omega^2 = \lambda - \frac14$$ 
It is known that the spectral gap of the double exponential distribution  on the whole real line is $\frac14$ \cite{Bobkov_Ledoux}.
By monotonicity of the spectral gap with respect to inclusion, we can deduce   that   $\lambda \ge \frac14$.
It seems clear that this inequality must be strict, and this can be proved directly.
Indeed, if $\lambda = \frac 14$, the general form of the solution is $f(t)=(At+B)e^{t/2}$.
Thus $f'(t) = \left( \frac{A}{2} t + \frac{B}{2} + A \right) e^{t/2}$
and the Neumann conditions lead to the linear system $Mu=0$ with:
$$M = \begin{pmatrix} 1 + \frac{a}2 & \frac12 \\ 1 + \frac{b}2 & \frac12 \end{pmatrix}
\quad \textrm{and} \quad u = \begin{pmatrix} A \\ B \end{pmatrix}$$
Since $a\neq b$, this implies $ A = B = 0$ which is impossible.\\
Hence $\lambda > \frac14$, and the solution has the form 
$f(t) = e^{t/2}(A \cos(\omega t) + B \sin(\omega t))$.
By applying Proposition~\ref{prop:exact_spectral_gap},
we deduce that the spectral gap is the first zero of 
$$ d(\lambda) = \cos(\omega a) \sin(\omega b) - \cos(\omega b) \sin(\omega a) = \sin(\omega (b-a))$$
Hence $\omega = \frac{\pi}{b-a}$ and 
$\lambda = \frac 14 + \left( \frac{\pi}{b-a} \right)^2$.\\

b) The same proof is immediately adapted to the case $a < b \leq 0$ by using the change of variables $x \mapsto -x$.\\

c) Let us now assume that $a < 0 < b$. 
By the same argument as in a), we know that $\lambda \geq \frac14$.
Let us temporarily admit that the case $\lambda = \frac14$ is impossible.
Then for $\lambda > \frac14$, the solution has the form:
$$f(x) = \left\{
  \begin{array}{rcr}
  e^{\vert x \vert/2}( A_{+} \cos(\omega x) + B_{+} \sin(\omega x) ) & \, \mathrm{if} \, x \geq 0 \\
  e^{\vert x \vert/2}( A_{-} \cos(\omega x) + B_{-} \sin(\omega x) ) & \, \mathrm{if} \, x \leq 0 \\
  \end{array}
\right. $$
where $\omega^2 = \lambda - \frac14$.
From Theorem~\ref{prop:SpectralTheorem} we know that $f$, searched in $\Sob{1}{\mu}((a,b))$, is actually of class $C^1$, and in particular $f$ and $f'$ must be continuous at $0$.
Conversely, if this condition is satisfied, $f$ will be of class $C^1$ since $f$ is $C^{\infty}$ on $[a, 0[$ and $]0, b]$. Hence $f$ will belong to $\Sob{1}{\mu}((a,b))$, as $\pdf$ is bounded on $[a,b]$.
Thus a necessary and sufficient condition for $f$ to belong to $\Sob{1}{\mu}((a,b))$ is that $f$ and $f'$ are continuous at $0$.\\
The first condition immediately implies $A_{+} = A_{-}$. For the second one, we have:
$$f'(x)e^{-\vert x \vert/2} = \left\{
  \begin{array}{rcrr}
  (A/2 + B_{+}\omega) \cos(\omega x) &+& (B_{+}/2 - A \omega) \sin(\omega x) & \, \mathrm{if} \, x \geq 0 \\
  (- A/2 + B_{-}\omega) \cos(\omega x) &+& (- B_{-}/2 - A \omega) \sin(\omega x) & \, \mathrm{if} \, x \leq 0 \\
  \end{array}
\right. $$
and continuity of $f'$ at $0$ leads to $B_{+} - B_{-} = - A / \omega$. 
Denoting $B = B_{-}$, we have $B_{+} = B - A / \omega$, and $f'$ is rewritten:
$$-f'(x)e^{-\vert x \vert/2} = (A/2 - B \omega) \cos(\omega x) +
\sin(\omega x) \times \left\{
  \begin{array}{rcl}
   -B/2 &+& A \left(\omega + \frac 1 {2\omega} \right) \\
   B/2 &+& A \omega \\
  \end{array}
\right. $$
The Neumann conditions $f'(a)=f'(b)=0$ then provide a linear system in $A, B$ which must be singular, 
leading to the condition $\det(M) = 0$ with:
$$ M = \begin{pmatrix} 
  \frac12 \cos(\omega b) +  \left(\omega + \frac 1 {2\omega} \right) \sin(\omega b)  & 
  -\omega \cos(\omega b) - \frac 12 \sin(\omega b)  \\ 
  \frac12 \cos(\omega a) +  \omega \sin(\omega a)  & 
  -\omega \cos(\omega a) + \frac 12 \sin(\omega a)  \\ 
\end{pmatrix}$$
Using the properties of the determinant, one can replace the first column $C_1$ by $C_1 + \frac1{2\omega}C_2$:
\begin{eqnarray*}
\det(M) &=& \det
\begin{pmatrix} 
 \left(\omega + \frac 1 {4\omega} \right) \sin(\omega b)  & 
  -\omega \cos(\omega b) - \frac 12 \sin(\omega b)  \\ 
 \left(\omega + \frac 1 {4\omega} \right) \sin(\omega a)  & 
  -\omega \cos(\omega a) + \frac 12 \sin(\omega a)  \\ 
\end{pmatrix} \\
&=& \left(\omega + \frac 1 {4\omega} \right) \sin(\omega a) \sin(\omega b) \det
\begin{pmatrix} 
 1  &   -\omega \cotan(\omega b) - \frac 12   \\ 
 1  &   -\omega \cotan(\omega a) + \frac 12   \\ 
\end{pmatrix} \\
&=& \omega \left(\omega + \frac 1 {4\omega} \right) \sin(\omega a) \sin(\omega b)
\big( \cotan(\omega \vert a \vert) + \cotan(\omega \vert b \vert) + 1/\omega \big) 
\end{eqnarray*}
The last factor in the above product is a decreasing function of  $\omega$ with unbounded limits in the interval $]0, c[$, where $c = \min \left(\pi/\vert a \vert, \pi/\vert b \vert \right)$. It has a unique zero on this interval, and it is clearly the 
 first non-negative zero of $\det(M)$ as a function of $\omega$. 
This proves the result for the Poincar\'e constant. 
The inequality $\omega > \frac \pi {b-a}$ is obtained by observing that 
$g: \omega \mapsto \cotan(\vert a \vert \omega ) + \cotan(\vert b \vert \omega )$ cancels at $\frac \pi {b-a} < c$
(for $a < 0 < b$) since $-a \frac \pi {b-a} = \pi - b \frac \pi {b-a}$. 
Thus if $\omega \leq \frac \pi {b-a}$ then 
$g(\omega)+\frac{1}{\omega} \geq \frac{1}{\omega} > 0$.\\
Finally the expression of $f$ is obtained by using the relation between $B_{+}$ and $B_{-}$, 
and the relation $M (A \, B)^T = 0$, which is equivalent (since $\det(M)=0$) to
$ A M_{21} + B M_{22} =0$. Recalling that $\sin(\omega a)$ does not vanish for the selected $\omega < \pi/\vert a \vert$, one can choose
$A = - M_{22}/\sin(\omega a) = \omega \cotan(\omega a) - \frac 12 $
and $B = M_{21}/\sin(\omega a) = \frac12 \cotan(\omega a) +  \omega $.\\

Finally, let us see that $\lambda=\frac14$ is impossible. Then the solution has the form:
$$f(x) = e^{\vert x \vert/2} \times \left\{
  \begin{array}{rccrr}
   A_{+} x &+& B_{+} & \, \mathrm{if} \, x \geq 0 \\
   A x &+& B & \, \mathrm{if} \, x \leq 0 \\
  \end{array}
\right. $$
As above, we express that $f$ and $f'$ are continuous at $0$.
The first condition implies $B_{+} = B$. For the second one, we have:
$$f'(x)e^{-\vert x \vert/2} = \left\{
  \begin{array}{rcrr}
  (A_{+}/2) x &+& A_{+} + B/2 & \, \mathrm{if} \, x \geq 0 \\
  (- A /2) x &+& A - B/2 & \, \mathrm{if} \, x \leq 0 \\
  \end{array}
\right. $$
and continuity of $f'$ at $0$ leads to $A_{+} =A - B$. Hence:
$$f'(x)e^{-\vert x \vert/2} = A - B/2 + \left\{
  \begin{array}{rr}
  (A/2 - B/2) x & \, \mathrm{if} \, x \geq 0 \\
  (-A/2) x  & \, \mathrm{if} \, x \leq 0 \\
  \end{array}
\right. $$
Now the Neumann conditions $f'(a)=f'(b)=0$ provide the linear system $M (A \, B)^T = 0$, with:
$$ M = \begin{pmatrix} 
  1-a/2  & -1/2\\ 
  1+b/2  & -1/2 - b/2 \\ 
\end{pmatrix}$$
This system must be singular otherwise $f$ would be identically zero, hence $\det(M)=0$.
However, $\det(M)=\frac14 (a - b + ab) < 0$ since $a<0<b$.
\end{proof}


\begin{proof}[Proof of Proposition~\ref{prop:triangular}]
For the triangular distribution $\mathcal{T}$ on $\Omega=(-1,1)$, we have: $\mu(dt) = \pdf(t)dt = e^{-V(t)}dt$ with $V(t) = - \ln(1-\vert t \vert)$. 
A simple application of Muckenhoupt criterion (\ref{eq:Muckenhoupt}) shows that $\mu$ admits a Poincar\'e constant. However, we cannot directly apply Theorem~\ref{prop:SpectralTheorem} since $\pdf$ vanishes at the boundary of $\Omega$. 
The idea is to consider an increasing sequence of subintervals $I_\epsilon$ such that $I_\epsilon \uparrow \Omega$ when $\epsilon \rightarrow 0$, for instance 
$I_\epsilon = (-1+\epsilon, 1-\epsilon)$.
Since $\mathcal{T}$ is a symmetric distribution on a symmetric support, 
a minimizer of the Rayleigh ratio can be found among odd functions (Lemma~\ref{lem:symmetry-quotient}).
Thus we can consider the spectral problem on $I_\epsilon^+ = (0, 1 - \epsilon)$:
$$ f''(t) - \frac{1}{1-t} f'(t) + \lambda f(t) = 0$$
with boundary conditions $f(0) = f'(1-\epsilon) = 0$.
By applying the change of variable $f(t) = y(x)$ with $x  = 1 - t$ , 
we obtain the equivalent problem on $(\epsilon, 1)$:
$$ y''(x) + \frac{1}{x} y'(x) + \lambda y(x) = 0$$
with initial conditions $ y(1) = y'(\epsilon) = 0$.
This is a Bessel-type differential equation, which solution has the form (see e.g. \cite{Abra_Steg}, Chapter 9):
$$ y(x) = A J_0 \left(x \sqrt{\lambda} \right) + B Y_0 \left(x \sqrt{\lambda} \right)$$
where $J_0$ and $Y_0$ are the Bessel functions of first and second type.
$J_0$ is an even function of class $C^{\infty}$ on $\R$, whereas $Y_0$ is  $C^{\infty}$ on $\R _+^*$ with infinite limit at $0^+$.
By using the same argument as in Proposition~\ref{prop:exact_spectral_gap}, $\sg(\mu_{\vert I_\epsilon})$ is then the first zero of 
$$ \tilde{d}(\lambda) = \det 
\begin{pmatrix} J_0 \left(\sqrt{\lambda} \right) & Y_0 \left(\sqrt{\lambda} \right) \\ 
J'_0 \left(\epsilon \sqrt{\lambda}\right) & Y'_0 \left(\epsilon \sqrt{\lambda} \right) \end{pmatrix} $$
Now $J'(0)=0$, but $Y'_0$ is not defined at $0$ and $x Y'_0(x) \rightarrow \frac{2}{\pi}$ as $x$ tends to zero (\cite{Abra_Steg}, Chapter 9, Eq. 9.1.9 and 9.1.28). So it is convenient to consider:
$$ d(\lambda, \epsilon) = \epsilon \tilde{d}(\lambda) = \det 
\begin{pmatrix} J_0 \left(\sqrt{\lambda} \right) & Y_0 \left(\sqrt{\lambda} \right) \\ 
\epsilon J'_0 \left(\epsilon \sqrt{\lambda} \right) & \epsilon Y'_0 \left(\epsilon \sqrt{\lambda} \right) \end{pmatrix}$$
By lemma~\ref{lem:continuitySupport}, 
$\sg(\mu) = \textrm{lim}_{\epsilon \rightarrow 0} \sg(\mu_{\vert I_\epsilon})$
and $d$ is continuous at $(\sg(\mu), 0)$.
Thus, it holds:
$$ 0 = d(\sg(\mu_{\vert I_\epsilon}), \epsilon) \underset{\epsilon \rightarrow 0}{\rightarrow} d(\sg(\mu), 0)$$
This proves that $\sg(\mu)$ is a zero of
$$ d(\lambda, 0) = \det 
\begin{pmatrix} J_0 \left(\sqrt{\lambda} \right) & Y_0 \left(\sqrt{\lambda} \right) \\ 
0 & \frac{2}{\pi \sqrt{\lambda}} \end{pmatrix} = \frac{2}{\pi \sqrt{\lambda}} J_0 \left(\sqrt{\lambda} \right).$$
Hence $\sqrt{\sg(\mu)}\ge r_1$, where $r_1$ is the first zero of $J_0$.\\

In order to prove the converse inequality, we compute the Rayleigh quotient of the following
odd function defined on $[-1,1]$: 
$$f(t) := \textrm{sign}(t) J_0 \left(r_1 ( 1 - \vert t \vert ) \right).$$
The expression of $f$ has been deduced from $y(x)=J_0(r_1x)$ by applying the transformation $x=1-t$, followed by a symmetrization.
By construction, $f$ and $f'$ are continuous at $0$. Thus $f$ is $C^1$ on $\overline{\Omega}$ and piecewise $C^2$ on $\Omega$. 
Furthermore, it verifies the following equation on $[0,1)$ 
\begin{equation} \label{eq:equadiffg}
f'' - V' f' + r_1^2 f = 0,
\end{equation}
with boundary conditions $f(0)=f'(1)=0$.
Since $\mu$ is even and $f$ is odd, the Rayleigh quotient of $f$ is simply $\int_0^1 (f')^2\rho/ \int_0^1 f^2\rho.$
In order to calculate this quotient, we multiply Eq.~\ref{eq:equadiffg} by $f \pdf$ and integrate by parts on $[0,1)$. 
However,  since $V = - \ln(\pdf)$ has a singularity at  $1$,  we rather fix $\epsilon > 0$ and work on  $I_\epsilon=[0,1-\epsilon]$.
Observe that $f$ and $f'\pdf$ are  $C^1$ on $I_\epsilon$. 
Starting from Eq.~\ref{eq:equadiffg} and then integrating by parts, we obtain:
$$ r_1^2 \int_{0}^{1-\epsilon} f^2 \pdf =-\int_0^{1-\epsilon}  (f'' - V'f')\pdf f =-\int_0^{1-\epsilon}  (\pdf f')' f =\int_{0}^{1-\epsilon} f'^2 \pdf
- \left[ f'f \pdf \right]_{0}^{1-\epsilon}.$$
The last term tends to $0$, when $\epsilon$ tends to $0$. Thus by monotone convergence,
$ r_1^2 \int_{0}^{1} f^2 \pdf= \int_{0}^{1} f'^2 \pdf$. We have shown that  the Rayleigh ratio of $f$ is equal to $r_1^2$. Therefore $\sg(\mu)\le r_1^2$ and  the proof is complete.  
\end{proof} 


\begin{proof}[Proof of Proposition~\ref{prop:SpectralGapTruncatedNormal}]
Here it is easier to consider the spectral problem with Dirichlet conditions (Problem (P4), see Prop. \ref{prop:NeumannEquivDirichlet}). It is given by:
\begin{equation}
\label{eq:normal_edo}
f''(t) - t f'(t) = - (\lambda - 1) f(t) \qquad s.t. \quad f(a) = f(b) = 0
\end{equation}
Let us look for a series expansion of $f$, i.e. $f(t) = \sum_{n \geq 0} c_n t^n$. Then,
\begin{eqnarray*}
f'(t) &=& \sum_{n \geq 1} n c_n t^{n-1} \quad  \Rightarrow \quad tf'(t) = \sum_{n \geq 0} n c_n t^{n}\\
f''(t) &=& \sum_{n \geq 2} n(n-1) c_n t^{n-2} = \sum_{n \geq 0} (n+2)(n+1) c_{n+2} t^{n}
\end{eqnarray*}
and solving (\ref{eq:normal_edo}) leads to finding a sequence $(c_n)$ satisfying:
\begin{equation}
\label{eq:normal_series}
c_{n+2} = \frac{n - (\lambda - 1)}{(n+1)(n+2)}c_n 
\end{equation}
Let us split $f$ into even an odd part, $f(t) = f_0(t) + t f_1(t)$ 
with $f_0(t) = \sum_{p \geq 0} u_{p}t^{2p}$ and $f_1(t) = \sum_{p \geq 0} v_{p}t^{2p}$,
where $u_p = c_{2p}$ and $v_p = c_{2p+1}$. 
Then by definition of Kummer series (\ref{hypergeom_form}), 
we recognize $f_0(t) = c_0 M_{\frac{1-\lambda}2; \frac12}\left(\frac{t^2}2\right)$
and $f_1(t) = c_1 M_{\frac{2-\lambda}2; \frac32}\left(\frac{t^2}2\right)$. 
Thus, with the notations of the Proposition,
$$ f(t) = c_0 h_0(\lambda, t) + c_1 t h_1(\lambda, t)$$
By applying Proposition~\ref{prop:exact_spectral_gap} and Remark~\ref{rem:prop:exact_spectral_gap},
we obtain that the spectral gap is the first zero of $d(\lambda)$ and hence the first part of the Proposition.
The symmetric case $a=-b$ is also deduced.\\
Furthermore:
\begin{itemize}
\item If there exists $\lambda$ such that $a, b$ are two successive zeros of $h_0(\lambda, .)$,   
 then up to a change of sign, $f = h_0(\lambda, .)$ is a positive solution satisfying $f(a)=f(b)=0$. Hence by Corollary~\ref{cor:sgWhenMonotoneSolution}, $\sg(\truncnorm) = \lambda$. Analogous conclusion stands with $h_1$. 
\item The last item is direct, using that, by construction, $h_0(., 0) = h_1(., 0) = 1$.
\end{itemize}
\end{proof}

\bibliographystyle{plain}
\bibliography{Poincare_Sensitivity.bib}

\end{document}